\documentclass[psamsfonts]{amsart}
\usepackage{amssymb,amsfonts}
\usepackage{latexsym}
\usepackage{enumerate}
\usepackage{mathrsfs}
\usepackage{url}
\usepackage{parsetree}

\newtheorem{thm}{Theorem}[section]
\newtheorem{cor}[thm]{Corollary}
\newtheorem{prop}[thm]{Proposition}
\newtheorem{lem}[thm]{Lemma}

\theoremstyle{definition}
\newtheorem{defn}[thm]{Definition}
\newtheorem{defns}[thm]{Definitions}

\theoremstyle{remark}

\makeatletter
\let\c@equation\c@thm
\makeatother
\numberwithin{equation}{section}

\title{Integer Complexity and Well-Ordering }
\author{Harry Altman}
\date{October 3, 2013}

\begin{document}

\newcommand{\cpx}[1]{\|#1\|}
\newcommand{\dft}{\delta}
\newcommand{\st}{{st}}
\newcommand{\xpdd}[1]{\hat{#1}}

\newcommand{\N}{{\mathbb N}}
\newcommand{\R}{{\mathbb R}}
\newcommand{\Z}{{\mathbb Z}}
\newcommand{\Q}{{\mathbb Q}}
\newcommand{\sS}{{\mathcal S}}

\newcommand{\floor}[1]{{\lfloor #1 \rfloor}}
\newcommand{\ceil}[1]{{\lceil #1 \ceil}}

\begin{abstract}
Define $\cpx{n}$ to be the \emph{complexity} of $n$, the smallest number of
ones needed to write $n$ using an arbitrary combination of addition and
multiplication.  John Selfridge showed that $\cpx{n}\ge 3\log_3 n$ for all $n$.
Define the \emph{defect} of $n$, denoted $\dft(n)$, to be $\cpx{n}-3\log_3 n$.
In this paper, we consider the set $\mathscr{D} := \{ \dft(n): n \ge 1 \}$ of
all defects.  We show that as a subset of the real numbers, the set
$\mathscr{D}$ is well-ordered, of order type $\omega^\omega$.  More
specifically, for $k\ge 1$ an integer, $\mathscr{D}\cap[0,k)$ has order type
$\omega^k$.  We also consider some other sets related to $\mathscr{D}$, and
show that these too are well-ordered and have order type $\omega^\omega$.
\end{abstract}

\maketitle

\section{Introduction}
\label{intro}

The \emph{complexity} of a natural number $n$ is the least number of $1$'s
needed to write it using any combination of addition and multiplication, with
the order of the operations specified using  parentheses grouped in any legal
nesting.  For instance, $n=11$ has a complexity of $8$, since it can be written
using $8$ ones as $(1+1+1)(1+1+1)+1+1$, but not with any fewer.  This notion was
implicitly introduced in 1953 by Kurt Mahler and Jan Popken \cite{MP}; they
actually considered the inverse function of the size of the largest number
representable using $k$ copies of the number $1$.  (More generally, they
considered the same question for representations using $k$ copies of a positive
real number $x$.) Integer complexity was explicitly studied by John Selfridge,
and was later popularized by Richard Guy \cite{Guy, UPINT}.  Following J. Arias
de Reyna \cite{Arias} we will denote the complexity of $n$ by $\cpx{n}$. 

Integer complexity is approximately logarithmic; it satisfies the bounds
\begin{equation}\label{eq1}
3 \log_3 n= \frac{3}{\log 3} \log  n\le \cpx{n} \le \frac{3}{\log 2} \log n  ,\qquad n>1.
\end{equation}
The lower bound can be deduced from the result of Mahler and Popken, and was
explicitly proved by John Selfridge \cite{Guy}. It is attained with equality for
$n=3^k$ for all $k \ge1$.  The upper bound can be obtained by writing $n$ in
binary and finding a representation using Horner's algorithm. It is not sharp,
and the constant $\frac{3}{\log2} $ can be improved for large $n$ \cite{upbds}.

The notion of integer complexity is similar in spirit but different in detail
from the better known measure of \emph{addition chain length}, which has
application to computation of powers, and which is discussed in detail in
Knuth \cite[Sect. 4.6.3]{TAOCP2}. One important difference between the two
notions is that integer complexity can be computed by dynamic programming,
while this does not seem to be the case for addition chain length.
Specifically, integer complexity is computable via the dynamic programming
recursion, for any $n>1$,

\begin{displaymath}
\cpx{n}=\min_{\substack{a,b<n\in \mathbb{N} \\ a+b=n\ \mathrm{or}\ ab=n}}
	\cpx{a}+\cpx{b}.
\end{displaymath}

There are many mysteries about $\cpx{n}$. For powers one has
\[ \cpx{n^k} \le k \cpx{n} \]
and it is known that $\cpx{3^k}= 3k$ for all $k \ge 1$.  However other values
have a more complicated behavior.  For instance, powers of $5$ do not work
nicely, as $\cpx{5^6}=29< 30= 6 \cpx{5}$. The behavior of powers of $2$ remains
unknown; it has been verified that 
\[
\cpx{2^k} = k \cpx{2} = 2k  ~~\mbox{for} ~~ 1\le k \le 39;
\] 
see \cite{data2}.  

\subsection{Main Result}
In an earlier paper, this author and Zelinsky \cite{paper1} introduced
the notion 
the \emph{defect} of an integer $n$, denoted $\dft(n)$, by
\[ \dft(n) := \cpx{n}-3\log_3 n. \]
This is a rescaled version of  integer complexity, which, given $n$, contains
equivalent information to $||n||$. In view of the lower bound \eqref{eq1} above
it satisfies $\dft(n)\ge 0$.  The paper \cite{paper1} exploited patterns in the
dynamic programming structure of integer complexity to classify the structure of
all integers with small values of the defect. In particular it classifies all
integers with $\dft(n) \le 1$.

The defect encodes interesting structure about integer complexity.  In this
paper, we will consider the image of this defect function in the general case:

\begin{defn}
The \emph{defect set} $\mathscr{D}\subseteq[0,\infty)$ is  the set
of all defect values
$\{\dft(n) : n\in \N\}$.
\end{defn}

Addition and multiplication tend to interact badly and unpredictably when
placed on an equal footing.  So one might  not expect to find any particular
sort of structure in the values of $\dft(n)$, even though its definition is
based on powers of $3$ which give the extremal case.  In this paper we will
prove the following striking result:

\begin{thm}
\label{frontpagethm}
The set $\mathscr{D}$ is a well-ordered subset of $\R$, of order type
$\omega^\omega$.  Furthermore, for $k\ge 1$ an integer, the set
$\mathscr{D}\cap[0,k)$ has order type $\omega^k$.
\end{thm}

This well-ordering of the defect set $\mathscr{D}$ reveals new fundamental
structure in the interaction between addition and multiplication.  Some of the
tangledness of that interaction may be reflected in how the set $\mathscr{D}$
grows more complicated as its elements get larger.  In fact the structure of
$\mathscr{D}$ has even more regularity than what Theorem \ref{frontpagethm}
describes, which we plan to discuss in a future paper. 

In Section~\ref{variants}, we will also prove that Theorem still holds even if
we replace $\mathscr{D}$ with any of several other closely-related sets.

Theorem~\ref{frontpagethm} is closely related to  conjectures of J.
Arias de Reyna \cite{Arias} about integer complexity.  
We discuss these conjectures and 
use our results to prove modified versions of 
some of them in Appendix~\ref{secarias}.  

In contrast to Theorem \ref{frontpagethm}, little is known about the set of
values of $\frac{\cpx{n}}{3\log_3 n}$, even though that might appear to be a
more natural object of study. An  open question is to determine
the value
\[
C_{max} := \limsup_{n \to \infty} \frac{\cpx{n}}{3\log_3 n}.
\]
The bounds \eqref{eq1} imply $1 \le C_{max} \le \log_2 3$.
It is an open problem to decide whether  $C_{max}=1$ or $C_{max}>1$ holds.

\subsection{Low-Defect Polynomials}

The strategy to prove the main theorem is to build up the set $\mathscr{D}$ by
inductively building up the sets $\mathscr{D}\cap[0,s)$ for real numbers $s>0$.
The proof of Theorem~\ref{frontpagethm} makes use of earlier work
of this author with Zelinsky \cite{paper1} classifying numbers of low defect.
The paper \cite{paper1} gave a method to list families of such integers, and
explicitly listed all integers of defect $\delta(n) < 1$. The innovation made
here is that instead of treating the output of this method as an
undifferentiated blob, we group it into tractable families.

We introduce a family of multilinear polynomials that we call \emph{low-defect
polynomials}. We show that for  any $s>0$, there exists a finite set of
low-defect polynomials $\sS_s$ such that any number of defect less than $s$ can
be written as $f(3^{n_1},\ldots,3^{n_k})3^{n_{k+1}}$ for some $f\in \sS_s$ and
nonnegative $n_1,\ldots,n_{k+1}$.  Indeed, stronger statements are true; see
Theorem~\ref{mainthm} and Theorem~\ref{augmainthm}.  Note, however, that the
low-defect polynomials may also produce extraneous numbers, with defect higher
than intended; examples of these are given after Theorem~\ref{mainthm}.  We will
remedy this deficiency in a sequel paper \cite{seq2}.

To state this another way, these low-defect polynomials provide forms into which
powers of $3$ can be substituted to obtain all the numbers below the specified
defect.  As the defects get larger, the low-defect polynomials and the families
of numbers we get this way become more complicated.  And just as we can
visualize expressions in $+$, $\times$, and $1$ as trees, we can also visualize
low-defect polynomials -- or the expressions that generate them -- as trees,
with open slots where powers of $3$ can be plugged in.  By attaching trees
corresponding to powers of $3$, we obtain trees for the numbers we get this way.
This is illustrated in Figure~\ref{treefig1} with the polynomial
$(2x_1+1)x_2+1$.  (Note, however, that this picture is not quite correct when we
plug in $3^0$; see Figure~\ref{treefig2} in section~\ref{secpolys}).
 
\begin{figure}[htb]
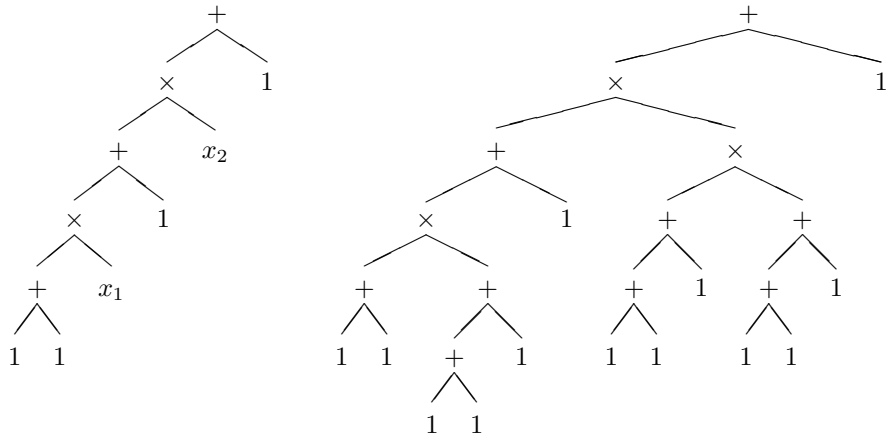

\caption{A tree corresponding to the polynomial $(2x_1+1)x_2+1$, and the same
tree after making the substitution $x_1=3^1$, $x_2=3^2$.}
\begin{tabular}{cc}
\label{treefig1}
\begin{parsetree}
( .$+$.
	( .$\times$.
		( .$+$.
			( .$\times$.
				( .$+$.
					`$1$'
					`$1$'
				)
				`$x_1$'
			)
			`$1$'
		)
		`$x_2$'
	)
	`$1$'
)
\end{parsetree}
&
\begin{parsetree}
( .$+$.
	( .$\times$.
		( .$+$.
			( .$\times$.
				( .$+$.
					`$1$'
					`$1$'
				)
				( .$+$.
					( .$+$.
						`$1$'
						`$1$'
					)
					`$1$'
				)
			)
			`$1$'
		)
		( .$\times$.
			(.$+$.
					( .$+$.
						`$1$'
						`$1$'
					)
				`$1$'
			)
			(.$+$.
					( .$+$.
						`$1$'
						`$1$'
					)
				`$1$'
			)
		)
	)
	`$1$'
)
\end{parsetree}
\end{tabular}
\end{figure}

So with this approach, we can get at properties of the set of defects by
examining properties of low-defect polynomials. For instance, as mentioned
above, as the defects involved get larger, the low-defect polynomials required
get more complicated; one way in which this occurs is that they require more
variables.  In fact, we will see (Theorem~\ref{mainthm}) that to cover defects
up to a real number $s$, one needs low-defect polynomials with up to $\lfloor s
\rfloor$ variables.  And it happens that if we have a low defect polynomial $f$
in $k$ variables, and consider the numbers $f(3^{n_1},\ldots,3^{n_k})$, then the
defects of the numbers obtained this way form a well-ordered set of order type
at least $\omega^k$ and less than $\omega^{k+1}$ (Proposition~\ref{wo2}).  It is
this that leads us to Theorem~\ref{frontpagethm}, that for $k\ge 1$, the set
$\mathscr{D}\cap[0,k)$ has order type precisely $\omega^k$.  In future
papers 
we will draw more detailed conclusions by examining the
structure of low-defect polynomials more closely.
\subsection{Variant Results}

We also prove analogues of the main theorem for several other sets.
The paper \cite{paper1} showed that given the value of $\dft(n)$, one can
determine the value of $\cpx{n}$ modulo $3$; see Theorem~\ref{oldprops}(6)
below.  It follows that one can split the set of defects $\mathscr{D}$ into sets
$\mathscr{D}^0$, $\mathscr{D}^1$, and $\mathscr{D}^2$ according to these
congruence classes modulo $3$; see Definition~\ref{dftseta}.  In
Section~\ref{variants} we prove analogues of the main theorem for each set
$\mathscr{D}^a$ separately; see Theorem~\ref{thmfinal}.

The paper \cite{paper1} also introduced a notion of stable numbers; a number
$n$ is said to be \emph{stable} if $\cpx{3^{k} n}=3k+\cpx{n}$ for all $k\ge 0$;
equivalently, if $\delta(3^{k}n) = \delta(n)$ for all $k \ge 0$.  In
Section~\ref{secstab} we show that given $\dft(n)$, one can determine whether or
not a given number $n$ is stable, and thus we can consider the set of ``stable
defects'', $\mathscr{D}_\st$, which are the defect values for all stable
numbers. 

We can combine this notion with splitting based on the value of $\cpx{n}$ modulo
$3$ to define sets $\mathscr{D}^0_\st$, $\mathscr{D}^1_\st$, and
$\mathscr{D}^2_\st$.  In Section~\ref{variants} we prove each of these sets is
well-ordered of type $\omega^{\omega}$, as are the closures of all these sets.
All these well-ordering results are collected in Theorem~\ref{thmfinal}.

\subsection{Computability Questions}

Integer complexity captures part of the complicated interaction of addition and
multiplication, where subtraction is not allowed; the underlying algebraic
structure is that of a commutative semiring $(\mathbb{N}, +, \times)$.  It is a
very simple computational model, but already exhibits difficult issues.

The model of computation treated in this paper could be considered as taking
number inputs other than $1$.  Mahler and Popken \cite{MP} considered
constructing numbers starting with copies of any fixed positive real number $x$.
Note that as $x$ varies the ordering of computed quantities on the positive real
line  will change.  One feature of complexity for $x=1$ (or for $x=k$, an
integer) is that multiple ties that can occur in doing the computations, which
complicates determination of the structure of the minimal computation tree.  For
a generic (transcendental) $x$, the complexity issue simplifies to viewing the
computation tree as computing a univariate polynomial with positive integer
coefficients, having a zero constant term.  One can assign a complexity to the
problem of computing such polynomials.  Study of this simplified problem might
be fruitful.
Allowing multiple indeterminates as inputs, we can consider the complexity of
computing multivariate polynomials, which is a much-studied topic.  The model of
computation allowing $+$ and $\times$ above can compute all multivariate
polynomials with nonnegative integer coefficients, but is restricted in that it
does not allow free reuse of polynomials already constructed. The complexity of
computation in this restricted model can be  compared to that in  other
computational models which allow additional operations beyond addition and
multiplication, or allow free reuse of already computed polynomials
(straight-line computation). 
It is much easier to compute polynomials in models with subtraction
\cite{valiant} or division \cite{FGK13} than with only addition and
multiplication \cite{burgbook, grigo, jersnir, schnorr}.  Indeed, similar
phenomena occur in the computation of integers as well as that of polynomials
\cite{borh}.

We can also ask about the computational complexity of integer complexity itself,
or related notions, viewed in the polynomial  hierarchy of complexity theory
(see Garey and Johnson \cite[Sect. 7.2]{GJ79}).  
An open question concerns the computational complexity of
computing $\cpx{n}$.  Consider the problem:\\

\noindent  INTEGER COMPLEXITY
\begin{itemize}
\item
INSTANCE: Positive integers $n$ and $k$, both encoded in binary.
\item
QUESTION: Is $\cpx{n}\le k$?
\end{itemize}
This problem is known to be in the complexity class $NP$ (Arias de Reyna
\cite{Arias}), but it  is not known to be either in $P$ or  in co-$NP$, nor 
is it known to be $NP$-complete.  

This paper introduces the ordering of defects as an object of
investigation.  Hence we can also consider the problem:\\

\noindent DEFECT ORDERING
\begin{itemize}
\item
INSTANCE: Positive integers $n_1$ and $n_2$, both encoded in binary.
\item
QUESTION: Is $\dft (n_1) \le \dft (n_2)$?
\end{itemize}
This problem, of computing the defect ordering  is not known
to be in the complexity class NP. If one could answer INTEGER COMPLEXITY  in polynomial time, then
one could also answer DEFECT ORDERING  in polynomial time.  To
show this, observe that the inequality $\dft(n_1)\le \dft(n_2)$ is equivalent to
\[ 3^{\cpx{n_1}}(n_2)^3 \le 3^{\cpx{n_2}}(n_1)^3, \]
and since $\cpx{n}$ is logarithmically small, this could be computed in
polynomial time if one knew $\cpx{n}$. This argument shows that
DEFECT ORDERING belongs to the complexity class  $P^{NP} = \Delta_2^P$.

Another question related to the defect is that of computing a set $\sS_s$ of
low-defect polynomials sufficient to describe all integers of defect
$\dft(n)<s$, i.e., a set $\sS_s$ satisfying the conditions of
Theorem~\ref{mainthm}.  What is the minimal cardinality of such a set, as a
function of $s$?  What is the complexity of computing one (say for $s$ integral,
or rational)?  The proof of Theorem~\ref{mainthm} does give a construction of
one such set $\sS_s$; however there exist other such sets $\sS_s$, perhaps some
smaller or computable more quickly than the one constructed.

\section{Properties of the defect}
\label{review}

We begin by reviewing the relevant properties of integer complexity and the
defect from \cite{paper1}.  They can be summed up in the following theorem:

\begin{thm}
\label{oldprops}
We have:
\begin{enumerate}
\item For all $n$, $\dft(n)\ge 0$.
\item For $k\ge 0$, $\dft(3^k n)\le \dft(n)$, with equality if and only if
$\cpx{3^k n}=3k+\cpx{n}$.  The difference $\dft(n)-\dft(3^k n)$ is a nonnegative integer.
\item If the difference $\dft(n)-\dft(m)$ is rational, then $n=m3^k$ for some
integer $k$ (and so $\dft(n)-\dft(m)\in\mathbb{Z}$).
\item Given any $n$, there exists $L$ such that for all $k\ge L$, $\dft(3^k
n)=\dft(3^L n)$.  That is to say, $\cpx{3^k n}=\cpx{3^L n}+3(k-L)$.
\item For a given defect $\alpha$, the set $\{m: \dft(m)=\alpha \}$ has either
the form $\{n3^k : 0\le k\le L\}$ for some $n$ and $L$, or the form $\{n3^k :
0\le k\}$ for some $n$.  This latter occurs if and only if $\alpha$ is the
smallest defect among $\dft(3^k n)$ for $k\in \mathbb{Z}$.
\item If $\dft(n)=\dft(m)$, then $\cpx{n}=\cpx{m} \pmod{3}$.
\item $\dft(1)=1$, and for $k\ge 1$, $\dft(3^k)=0$.  No other integers occur as
$\dft(n)$ for any $n$.
\end{enumerate}
\end{thm}

\begin{proof}
Part (1) is just Selfridge's lower bound \cite{Guy}.  The first statement in
part (2) is Proposition~9(3) from \cite{paper1}; the second statement follows
from the computation $\dft(n)-\dft(3^k n)= \cpx{n}-\cpx{3^k n}+3k$.  Part (3)
is Proposition~14(1) from \cite{paper1}.  Parts (4) and (5) are Theorem~5
from \cite{paper1}.  Part (6) is part of Proposition~14(2) from \cite{paper1}.
For part (7), the fact that $\dft(1)=1$ is immediate.  The fact that
$\dft(3^k)=0$ for $k\ge 1$ is the same as the fact that $\cpx{3^k}=3k$ for
$k\ge 1$; that $\cpx{3^k}\le 3k$ is obvious, and that $\cpx{3^k}\ge 3k$ follows
from Selfridge's lower bound \cite{Guy}.  Finally, that no other integers occur
as $\dft(n)$ for any $n$ follows from part (3).
\end{proof}

We also recall the definitions made for discussing the above:

\begin{defn}
A number $m$ is called \emph{stable} if $\cpx{3^k m}=3k+\cpx{m}$ holds for every
$k \ge 1$, or equivalently if $\dft(3^k m)=\dft(m)$ for every $k\ge 1$.
Otherwise it is called \emph{unstable}.
\end{defn}

\begin{defn}
A natural number $n$ is called a \emph{leader} if it is the smallest number with
a given defect.  By part (5) of Theorem~\ref{oldprops}, this is equivalent to
saying that either $3\nmid n$, or, if $3\mid n$, then $\dft(n)<\dft(n/3)$, i.e.,
$\cpx{n}<3+\cpx{n/3}$.
\end{defn}

Also, because of part (6) of Theorem~\ref{oldprops}, we can make the following
definitions:

\begin{defn}
\label{dftseta}
For $a$ a congruence class modulo $3$, we define
\[ \mathscr{D}^a = \{ \dft(n) : \cpx{n}\equiv a\pmod{3},~~n\ne 1\} \]
\end{defn}

We explicitly exclude the number $1$ here as it is dissimilar to other numbers
whose complexity is congruent to $1$ modulo $3$.  This is because, unlike other
numbers which are $1$ modulo $3$, the number $1$ cannot be written as $3j+4$ for
some $j$, and so the largest number that can be made with a single $1$ is simply
$1$, rather than $4\cdot 3^j$ (see Appendix~\ref{secarias}).  For this reason,
numbers of complexity $1$ do not really go together with other numbers whose
complexity is congruent to $1$ modulo $3$; however, the only such number is $1$,
so we simply explicitly exclude it.  So $\mathscr{D}$ is the disjoint union of
$\mathscr{D}^0$, $\mathscr{D}^1$, $\mathscr{D}^2$, and $\{1\}$.

Of course, we care not just about small defects, but about the numbers giving
rise to those small defects; so we recall the following definitions:

\begin{defn}
For any real $r\ge0$, define the set of {\em $r$-defect numbers} $A_r$ to be 
$$A_r := \{n\in\mathbb{N}:\dft(n)<r\}.$$
Define the set of {\em $r$-defect leaders} $B_r$ to be 
$$
B_r:= \{n \in A_r :~~n~~\mbox{is a leader}\}.
$$
\end{defn}

These sets are related by:
\begin{prop}
\label{arbr}
For every $n\in A_r$, there exists a unique $m\in B_r$ and $k\ge 0$ such that
$n=3^k m$ and $\dft(n)=\dft(m)$; then $\cpx{n}=\cpx{m}+3k$.
\end{prop}

\begin{proof}
The first part of this is Proposition 16(2) from \cite{paper1}.  The second
part follows as then
\[ \cpx{n}=\dft(n)+3\log_3(3^k m)=3k+\dft(m)+3\log_3 m =\cpx{m}+3k.\]
\end{proof}

\subsection{Inductive covering of $B_r$ and $A_r$}

In addition to the above properties of the defect, there are two substantive
theorems we will need from \cite{paper1}.  They allow us to inductively build up
the sets $A_r$ and $B_r$, or at least coverings of these.  The first provides
the base case:

\begin{thm}
\label{finite}
For every $\alpha$ with $0<\alpha<1$, the set of leaders $B_\alpha$ is a finite
set.
\end{thm}

The other theorem provides the inductive step, telling us how to build up
$B_{(k+1)\alpha}$ from previous $B_{i\alpha}$.  In order to state it we'll first
need some definitions.

\begin{defns}
We say $n$ is \emph{most-efficiently} represented as $ab$ if $n=ab$ and
$\cpx{n}=\cpx{a}+\cpx{b}$, or as $a+b$ if $n=a+b$ and $\cpx{n}=\cpx{a}+\cpx{b}$.
In the former case we will also say that $n=ab$ is a \emph{good factorization}
of $n$.  We say $n$ is \emph{solid} if it cannot be written most-efficiently as
$a+b$ for any $a$ and $b$.  We say $n$ is \emph{m-irreducible} if it cannot be
written most-efficiently as $ab$ for any $a$ and $b$.  And for a real number
$\alpha\in(0,1)$, we define the set $T_\alpha$ to consist of $1$ together with
those $m$-irreducible numbers $n$ which satisfy
\[\frac{1}{n-1}>3^{\frac{1-\alpha}{3}}-1\]
and do not satisfy $\cpx{n}=\cpx{n-b}+\cpx{b}$ for any solid numbers $b$ with
$1<b\le n/2$.
\end{defns}

Note that for any $0<\alpha<1$, the set $T_\alpha$ is a finite set, due to the
upper bound on the size of numbers $n\in T_\alpha$.

Now we can state the theorem.  The theorem provides fives possibilities; three
``generic cases'' (1 through 3), and two ``exceptional cases'' (4 and 5).

\begin{thm}
\label{themethod}
Suppose that $0< \alpha <1$ and that $k\ge1$.
Then any $n\in B_{(k+1)\alpha}$ can be most-efficiently
represented in (at least) one of the following forms:
\begin{enumerate}
\item
For $k=1$,
there is either a good factorization $n=u\cdot v$ where
$u,v\in {B}_\alpha$, or a good factorization $n=u\cdot v\cdot w$ with
$u,v,w\in {B}_\alpha$; \\
For $k \ge 2$, there is a good factorization $n=u \cdot v$ where $u\in
B_{i\alpha}$,
$v\in B_{j\alpha}$ with $i+j=k+2$ and $2\le i, j\le k$.
\item $n=a+b$ with $\cpx{n}=\cpx{a}+\cpx{b}$, $a\in A_{k\alpha}$, $b\le a$ a
solid number and
\[\dft(a)+\cpx{b}<(k+1)\alpha+3\log_3 2.\]
\item There is a good factorization $n=(a+b)v$ with $v\in B_\alpha$, $a+b$ being
a most-efficient representation, and $a$ and $b$ satisfying the conditions in
the case (2) above.
\item $n\in T_\alpha$ (and thus in particular either $n=1$ or
$\cpx{n}=\cpx{n-1}+1$.)
\item There is a good factorization $n = u\cdot v$ with $u\in T_\alpha$ and
$v\in B_\alpha$.
\end{enumerate}
\end{thm}

By applying these two theorems, we can inductively build up the sets $B_r$ and
$A_r$; in a sense they form the engine of our proof.  However, without
additional tools, it can be hard to say anything about just what these theorems
output.  In Section~\ref{secpolys}, we will show how to group the output of
these theorems into tractable families, allowing us to go beyond the earlier
work of this author and Zelinsky \cite{paper1} and prove the main theorem.

\section{Stable defects and stable complexity}
\label{secstab}

It will also be useful here to introduce the notion of ``stable defect'' and
``stable complexity''.  First, let us discuss the defects of stable numbers.

\begin{prop}
If $\dft(n)=\dft(m)$ and $n$ is stable, then so is $m$.
\end{prop}

\begin{proof}
Suppose $\dft(n)=\dft(m)$ and $n$ is stable.  Then we can write $m=3^k n$ for
some $k\in \mathbb{Z}$.  Now, a number $a$ is stable if and only if $\dft(3^\ell
a)=\dft(a)$ for all $\ell\ge 0$; so if $k\ge 0$, then $m$ is stable.  If, on the
other hand, $k<0$, then consider $\ell\ge 0$.  If $\ell\ge -k$, then
$\dft(3^\ell m)=\dft(3^{\ell+k} n)=\dft(n)$, while if $\ell\le -k$, then
$\dft(n)\le \dft(3^\ell m)\le \dft(m)$, so $\dft(3^\ell m)=\dft(m)$; hence $m$
is stable.
\end{proof}

Because of this proposition, it makes sense to make the following definition:

\begin{defn}
We define a \emph{stable defect} to be the defect of a stable number, and define
$\mathscr{D}_\st$ to be the set of all stable defects.  Also, for $a$ a
congruence class modulo $3$, we define $\mathscr{D}^a_\st=\mathscr{D}^a \cap
\mathscr{D}_\st$.
\end{defn}

Note that the integer $1$ is not stable, and so its defect, which is also $1$,
would be excluded from $\mathscr{D}^1_\st$ even if we had not explicitly
excluded it in the definition of $\mathscr{D}^1$.

This double use of the word ``stable'' could potentially be ambiguous if we had
a positive integer $n$ which were also a defect.  However, the only positive
integer which is also a defect is $1$, which is not stable in either sense.

\begin{prop}
\label{modz1}
A defect $\alpha$ is stable if and only if it is the smallest
$\beta\in\mathscr{D}$ such that $\beta\equiv\alpha\pmod{1}$.
\end{prop}

\begin{proof}
This follows from parts (2), (3), and (5) of Theorem~\ref{oldprops}.
\end{proof}

\begin{defn}
For a positive integer $n$, define the \emph{stable defect of $n$}, denoted
$\dft_\st(n)$, to be $\dft(3^k n)$ for any $k$ such that $3^k n$ is stable.
(This is well-defined as if $3^k n$ and $3^\ell n$ are stable, then $k\ge \ell$
implies $\dft(3^k n)=\dft(3^\ell n)$, and so does $\ell\ge k$.)
\end{defn}

Here are two equivalent characterizations:

\begin{prop}
\label{staltchar}
The number $\dft_\st(n)$ can be characterized by:
\begin{enumerate}
\item $\dft_\st(n)= \min_{k\ge 0} \dft(3^k n)$
\item $\dft_\st(n)$ is the smallest $\alpha\in\mathscr{D}$ such that
$\alpha\equiv \dft(n) \pmod{1}$.
\end{enumerate}
\end{prop}

\begin{proof}
Part (1) follows from part (2) Theorem~\ref{oldprops} and the fact that $m$ is
stable if and only if $\dft(3^k m)=\dft(m)$ for all $k\ge 0$.  To prove part
(2), take $k$ such that $3^k n$ is stable.  Then $\dft(3^k n)\equiv \dft(n)
\pmod{1}$, and it is the smallest such by Proposition~\ref{modz1}.
\end{proof}

So we can think about $\mathscr{D}_\st$ either as the subset of $\mathscr{D}$
consisting of the stable defects, or we can think about it as the image of
$\dft_\st$.  (This latter way of thinking doesn't work so well for the
$\mathscr{D}_\st^a$, however.)

Just as we can talk about the stable defect of a number $n$, we can also talk
about its \emph{stable complexity} -- what the complexity would be ``if $n$ were
stable''.

\begin{defn}
For a positive integer $n$, we define the \emph{stable complexity of $n$},
denoted $\cpx{n}_\st$, to be $\cpx{3^k n}-3k$ for any $k$ such that $3^k n$ is
stable.  This is well-defined; if $3^k n$ and $3^\ell n$ are both stable, say
with $k\le \ell$, then
\[\cpx{3^k n}-3k=3(k-\ell)+\cpx{3^\ell n}-3k=\cpx{3^\ell n}-3\ell.\]
\end{defn}

\begin{prop}
\label{stabcpxprops}
We have:
\begin{enumerate}
\item $\cpx{n}_\st = \min_{k\ge 0} (\cpx{3^k n}-3k)$
\item $\dft_\st(n)=\cpx{n}_\st-3\log_3 n$
\end{enumerate}
\end{prop}

\begin{proof}
To prove part (1), observe that $\cpx{3^k n}-3k$ is nonincreasing in $k$, since
$\cpx{3m}\le3+\cpx{m}$.  So a minimum is achieved if and only if for all $\ell$,
\[\cpx{3^{k+\ell} n}-3(k+\ell)=\cpx{3^k n}-3k,\] i.e., for all $\ell$,
$\cpx{3^{k+\ell} n}=\cpx{3^k n}+3\ell$, i.e., $3^k n$ is stable.

To prove part (2), take $k$ such that $3^k n$ is stable.  Then
\[\dft_\st(n)=\dft(3^k n)=\cpx{3^k n}-3\log_3(3^k n)=\cpx{3^k n}-3k-3\log_3 n
=\cpx{n}_\st-3\log_3 n.\]
\end{proof}

\begin{prop}
\label{stabisstab}
We have:
\begin{enumerate}
\item $\dft_\st(n) \le \dft(n)$, with equality if and only if $n$ is stable.
\item $\cpx{n}_\st \le \cpx{n}$, with equality if and only if $n$ is stable.
\end{enumerate}
\end{prop}

\begin{proof}
The inequality in part (1) follows from Proposition~\ref{staltchar}.  Also, if
$n$ is stable, then for any $k\ge 1$, we have $\dft(3^k n)=\dft(n)$, so
$\dft_\st(n)=\dft(n)$.  Conversely, if $\dft_\st(n)=\dft(n)$, then by
Proposition~\ref{staltchar}, for any $k\ge 1$, we have $\dft(3^k n)\ge \dft(n)$.
But also $\dft(3^k n)\le \dft(n)$ by part (2) of Theorem~\ref{oldprops}, and so
$\dft(3^k n)=\dft(n)$ and $n$ is stable.

Part (2) follows from part (1) along with part (2) of
Proposition~\ref{stabcpxprops}.
\end{proof}

We will write more about the properties of $\cpx{n}_\st$ in a sequel paper
\cite{seq2}.

\section{Low-defect polynomials}
\label{secpolys}

The primary tool we will use to prove the main theorem is to group the numbers
produced by the main theorem of \cite{paper1} into families.  Each of these
families will be expressed via a multilinear polynomial in $\Z[x_1,x_2,\ldots]$,
which we will call a \emph{low-defect polynomial}.  We will associate these with
a ``base complexity'' to form a \emph{low-defect pair}.  Formally:

\begin{defn}
We define the set $\mathscr{P}$ of \emph{low-defect pairs} as the smallest
subset of $\Z[x_1,x_2,\ldots]\times \N$ such that:
\begin{enumerate}
\item For any constant polynomial $k\in \N\subseteq\Z[x_1, x_2, \ldots]$ and any
$C\ge \cpx{k}$, we have $(k,C)\in \mathscr{P}$.
\item Given $(f_1,C_1)$ and $(f_2,C_2)$ in $\mathscr{P}$, we have $(f_1\otimes
f_2,C_1+C_2)\in\mathscr{P}$, where, if $f_1$ is in $r_1$ variables and $f_2$ is
in $r_2$ variables,
\[ (f_1\otimes f_2)(x_1,\ldots,x_{r_1+r_2}) :=
	f_1(x_1,\ldots,x_{r_1})f_2(x_{r_1+1},\ldots,x_{r_1+r_2}). \]
\item Given $(f,C)\in\mathscr{P}$, $c\in \N$, and $D\ge \cpx{c}$, we have
$(f\otimes x_1 + c,C+D)\in\mathscr{P}$ where $\otimes$ is as above.
\end{enumerate}

The polynomials obtained this way will be referred to as \emph{low-defect
polynomials}.  If $(f,C)$ is a low-defect pair, $C$ will be called its
\emph{base complexity}.  If $f$ is a low-defect polynomial, we will define its
\emph{absolute base complexity}, denoted $\cpx{f}$, to be the smallest $C$ such
that $(f,C)$ is a low-defect pair.
\end{defn}

Note that the degree of a low-defect polynomial is also equal to the number of
variables it uses; see Proposition~\ref{polystruct}.  We will often refer to the
``degree'' of a low-defect pair $(f,C)$; this refers to the degree of $f$.

Note that we do not really care about what variables a low-defect polynomial (or
pair) is in -- if we permute the variables of a low-defect polynomial or replace
them with others, we will still regard the result as a low-defect polynomial.
From this perspective, the meaning of $f\otimes g$ could be simply regarded as
``relabel the variables of $f$ and $g$ so that they do not share any, then
multiply $f$ and $g$''.  Helpfully, the $\otimes$ operator is associative not
only with this more abstract way of thinking about it, but also in the concrete
way it was defined above.

\subsection{Properties of low-defect polynomials}

Let us begin by stating some structural properties of low-defect polynomials.

\begin{prop}
\label{polystruct}
Suppose $f$ is a low-defect polynomial of degree $r$.  Then $f$ is a
polynomial in the variables $x_1,\ldots,x_r$, and it is a multilinear
polynomial, i.e., it has degree $1$ in each of its variables.  The coefficients
are non-negative integers.  The constant term is nonzero, and so is the
coefficient of $x_1\ldots x_r$, which we will call the \emph{leading
coefficient} of $f$.
\end{prop}

\begin{proof}
We prove the statement by structural induction.

If the low-defect polynomial $f$ is just a constant $n$, it has no variables and
the leading coefficient and constant term are both $n$, which is positive.

If $f=g\otimes h$, say
$f(x_1,\ldots,x_r)=g(x_1,\ldots,x_s)h(x_{s+1},\ldots,x_r)$, then by the
inductive hypothesis $f$ is a product of two polynomials whose coefficients are
nonnegative integers, and thus so is $f$.  To see that $f$ is multilinear,
consider a variable $x_i$; if $1\le i\le s$, then $x_i$ has degree $1$ in
$g(x_1,\ldots,x_s)$ and degree $0$ in $h(x_{s+1},\ldots,x_r)$, while if $r+1 \le
i \le s$, the reverse is true.  Either way, $x_i$ has degree $1$ in $f$.

The coefficient of $x_1\ldots x_r$ in $f$ is the product of the coefficient of
$x_1\ldots x_s$ in $g$ and the coefficient of $x_1\ldots x_{r-s}$ in $h$ and so
does not vanish, and the constant term of $f$ is the product of the constant
terms of $g$ and $h$ and so does not vanish.

Finally, if $f=g\otimes x_1 + c$, say
$f(x_1,\ldots,x_r)=g(x_1,\ldots,x_{r-1})x_r+c$, then since $g$ has coefficients
that are nonnegative integers, so does $f$.  To see that $f$ is multilinear,
consider a variable $x_i$; for $1\le i \le r-1$, the variable $x_i$ has degree
$1$ in $g$ and hence so does in $f$, while $x_r$ has degree $0$ in $g$ and hence
has degree $1$ in $f$ as well.  Finally, the coefficient of $x_1\ldots x_r$ in
$f$ is the same as the coefficient of $x_1\ldots x_{r-1}$ in $g$ and hence does
not vanish, while the constant term of $f$ is $c$, which is positive.
\end{proof}

We will also need the following lemma in Section~\ref{secwo}:

\begin{lem}
\label{maxvar}
For any low-defect polynomial $f$ of degree $k>0$, there exist low-defect
polynomials $g$ and $h$ and a positive integer $c$ such that
$f=h\otimes(g\otimes x_1+c)$.
\end{lem}

\begin{proof}
We apply structural induction.  Since $f$ has degree greater than zero, it is
not a constant.  Hence either it can be written as $f_1 \otimes f_2$ (in which
case at least one of these has degree greater than zero) or as $g\otimes x_1 +
c$.  In the latter case we are done, writing $f=1\otimes(g\otimes x_1+c)$.

In the former case, without loss of generality, say $f_2$ has degree $r>0$.
(Since if $f_2$ is a constant, $f_1\otimes f_2=f_2 \otimes f_1$.) Then by the
inductive hypothesis, there are low-defect polynomials $g_2$ and $h_2$ and a
positive integer $c_2$ such that $f_2=h_2\otimes(g_2 \otimes x_1+c)$, so
$f=(f_1\otimes h_2)\otimes(g_2\otimes x_1+c)$, as needed.
\end{proof}

There is more that can be said about the structure of low-defect polynomials, as
we will show in \cite{seq2}.

\subsection{Numbers $3$-represented by low-defect polynomials}
\label{3rep}

We will obtain actual numbers from these polynomials by substituting in powers
of $3$ as mentioned in Section~\ref{intro}.  Let us state here the following
obvious but useful lemma:

\begin{lem}
\label{powers}
For any $a, b$, and $n$, $\cpx{ab^n}\le\cpx{a}+n\cpx{b}$.
\end{lem}

\begin{proof}
If $n\ge 1$, then $\cpx{ab^n}\le\cpx{a}+\cpx{b^n}\le\cpx{a}+n\cpx{b}$.  Whereas
if $n=0$, then $\cpx{ab^n}=\cpx{a}=\cpx{a}+n\cpx{b}$.
\end{proof}

This provides an upper bound on the complexities of the outputs of these
polynomials:

\begin{prop}
\label{basicub}
If $(f,C)$ is a low-defect pair of degree $r$, then
\[\cpx{f(3^{n_1},\ldots,3^{n_r})}\le C+3(n_1+\ldots+n_r).\]
\end{prop}

\begin{proof}
We prove the statement by structural induction.  If $f$ is a constant $k$, then
$C\ge\cpx{k}$, and we are done.

If there are low-defect pairs $(g_1,D_1)$ and $(g_2,D_2)$ (say of degrees $s_1$
and $s_2$) such that $f=g_1\otimes g_2$ and $C=D_1+D_2$, then
\begin{eqnarray*}
\cpx{f(3^{n_1},\ldots,3^{n_r})} \le
\cpx{g_1(3^{n_1},\ldots,3^{n_{s_1}})}+\cpx{g_2(3^{n_{s_1+1}},\ldots,3^{n_r})}
\\ \le D_1+D_2+3(n_1+\ldots+n_r) = C+3(n_1+\ldots+n_r).
\end{eqnarray*}

In the last case, if there is a low-defect pair $(g,D)$ and a constant $c$ with
$C\ge D+\cpx{c}$ such that $f=g\otimes x_1 + c$, we apply
Lemma~\ref{powers}:
\begin{eqnarray*}
\cpx{f(3^{n_1},\ldots,3^{n_r})} \le
\cpx{g(3^{n_1},\ldots,3^{n_{r-1}})}+3n_r+\cpx{c} \\ \le
D+\cpx{c}+3(n_1+\ldots+n_r) \le C+3(n_1+\ldots+n_r).
\end{eqnarray*}

\end{proof}

Note that because of the two cases in the proof of Lemma~\ref{powers}, the
picture in Figure~\ref{treefig1} is slightly inaccurate; this is only the
picture when $3^k$ is plugged in for $k\ge 1$.  See Figure~\ref{treefig2} for an
illustration of what happens when we plug in $3^0$.

\begin{figure}
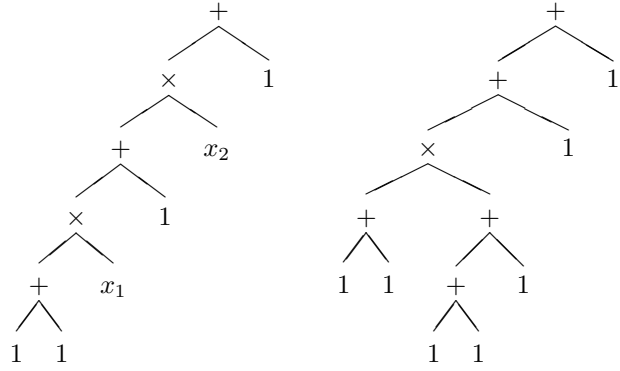

\caption{A tree corresponding to the polynomial $(2x_1+1)x_2+1$, and the same
tree after making the substitution $x_1=3^1$, $x_2=3^0$; observe how the top
multiplication node disappears.}
\begin{tabular}{cc}
\label{treefig2}
\begin{parsetree}
( .$+$.
	( .$\times$.
		( .$+$.
			( .$\times$.
				( .$+$.
					`$1$'
					`$1$'
				)
				`$x_1$'
			)
			`$1$'
		)
		`$x_2$'
	)
	`$1$'
)
\end{parsetree}
&
\begin{parsetree}
( .$+$.
	( .$+$.
		( .$\times$.
			( .$+$.
				`$1$'
				`$1$'
			)
			( .$+$.
				( .$+$.
					`$1$'
					`$1$'
				)
				`$1$'
			)
		)
		`$1$'
	)
	`$1$'
)
\end{parsetree}
\end{tabular}
\end{figure}

Because of Proposition~\ref{basicub}, we define:

\begin{defn}
Given a low-defect pair $(f,C)$ (say of degree $r$) and a number $N$, we will
say that $(f,C)$ \emph{efficiently $3$-represents} $N$ if there exist
nonnegative integers $n_1,\ldots,n_r$ such that $N=f(3^{n_1},\ldots,3^{n_r})$
and $\cpx{N}=C+3(n_1+\ldots+n_r)$.
More generally, we will also say $f$ $3$-represents $N$ if there exist
nonnegative integers $n_1,\ldots,n_r$ such that $N=f(3^{n_1},\ldots,3^{n_r})$.
\end{defn}

Note that if $(f,C)$ efficiently $3$-represents $N$, then $(f,\cpx{f})$
efficiently $3$-represents $N$, which means that in order for $(f,C)$ to
$3$-represent anything efficiently at all, we must have $C=\cpx{f}$.  However it
is still worth using low-defect pairs rather than just low-defect polynomials
since we may not always know $\cpx{f}$.  This paper will not be concerned with
these sorts of computational issues, but in a future paper \cite{seq2} we will
discuss how to refine the theorems here to allow for computation.

For this reason it makes sense to use ``$f$ efficiently $3$-represents $N$'' to
mean ``some $(f,C)$ efficiently $3$-represents $N$'' or equivalently
``$(f,\cpx{f})$ efficiently $3$-reperesents $N$''.

In keeping with the name, the numbers $3$-represented by a low-defect polynomial
have bounded defect.  First let us make two definitions:

\begin{defn}
Given a low-defect pair $(f,C)$, we define $\dft(f,C)$, the defect of $(f,C)$,
to be $C-3\log_3 a$, where $a$ is the leading coefficient of $f$.  When we are
not concerned with keeping track of base complexities, we will use $\dft(f)$ to
mean $\dft(f,\cpx{f})$.
\end{defn}

\begin{defn}
Given a low-defect pair $(f,C)$ of degree $r$, we define
\[\dft_{f,C}(n_1,\ldots,n_r) =
C+3(n_1+\ldots+n_r)-3\log_3 f(3^{n_1},\ldots,3^{n_r}).\]
We will also define $\dft_f$ to mean $\dft_{f,\cpx{f}}$ when we are not
concerned with keeping track of base complexities.
\end{defn}

Then we have:

\begin{prop}
\label{dftbd}
Let $(f,C)$ be a low-defect pair of degree $r$, and let $n_1,\ldots,n_r$ be
nonnegative integers.
\begin{enumerate}
\item We have
\[ \dft(f(3^{n_1},\ldots,3^{n_r}))\le \dft_{f,C}(n_1,\ldots,n_r)\]
and the difference is an integer.
\item We have \[\dft_{f,C}(n_1,\ldots,n_r)\le\dft(f,C)\]
and if $r\ge 1$, this inequality is strict.
\end{enumerate}
\end{prop}

\begin{proof}
For part (1), observe that this inequality is just Proposition~\ref{basicub}
with the quantity $3\log_3(f(3^{n_1},\ldots,3^{n_r})$ subtracted off both sides.
And since Proposition~\ref{basicub} is an inequality of integers, the difference
is an integer.

For part (2), let $a$ denote the leading coefficient of $f$.  Then by
Proposition~\ref{polystruct},
\[ f(3^{n_1},\ldots,3^{n_r})\ge a\cdot3^{n_1+\ldots+n_r}, \]
and this inequality is strict if $r\ge1$ (since
the constant term of $f$ does not vanish).  So
\begin{eqnarray*}
\dft_{f,C}(n_1,\ldots,n_r) =
C + 3(n_1+\ldots+n_r) -3\log_3 f(3^{n_1},\ldots,3^{n_r})\\
\le C+3(n_1+\ldots+n_r) - 3\log_3(a) - 3(n_1+\ldots+n_r) \\
= C - 3\log_3(a)= \delta(f,C),
\end{eqnarray*}
and this inequality is strict if $r\ge 1$.
\end{proof}

\subsection{Low-defect polynomials give all leaders of small defect}
\label{mainthmsec}

The reason these polynomials are relevant is as follows:

\begin{thm}
\label{mainthm}
For any real $r\ge 0$, there exists a finite set $\sS_r$ of low-defect pairs
satisfying the following conditions:
\begin{enumerate}
\item Each $(f,C)\in \sS_r$ has degree at most $\lfloor r \rfloor$;
\item  For every $N\in B_r$, there exists some $(f,C)\in \sS_r$ that efficiently
$3$-represents $N$.
\end{enumerate}
\end{thm}

\begin{proof}
We prove this statement in the following form: For any real $\alpha\in(0,1)$ and
any integer $k\ge 1$, there exists a finite set $\sS_{k,\alpha}$ of low-defect
pairs, each of degree at most $k-1$, such that for every $N\in B_{k\alpha}$
there exists some $(f,C)\in \sS_{\alpha,r}$ that efficiently $3$-represents $N$.
Once we have this, the result will follow by taking $\sS_r=\sS_{k,\alpha}$ for
$k=\lfloor r \rfloor+1$ and $\alpha=\frac{r}{\lfloor r \rfloor + 1}$.

We prove this by induction on $k$.  If $k=1$, then $B_\alpha$ is finite by
Theorem~\ref{finite}, so we can take $\sS_{1,\alpha}=\{(N,\cpx{N}) : N\in
B_\alpha \}$.  Now suppose the statement is true for $k$, and we want to prove
it for $k+1$, so we have already constructed sets $\sS_{i,\alpha}$ for $i\le k$.
   
We will define the set $\sS_{k+1,\alpha}$ to consist of the following:

\begin{enumerate}
\item If $k+1>2$, then for $(f,C)\in \sS_{i,\alpha}$ and $(g,D)\in
\sS_{j,\alpha}$ with $2\le i, j\le k$ and $i+j=k+2$ we include $(f\otimes g,
C+D)$ in $\sS_{k+1,\alpha}$; \\
while if $k+1=2$, then for $(f_1,C_1),(f_2,C_2),(f_3,C_3)\in \sS_{1,\alpha}$,
we include $(f_1 \otimes f_2, C_1+C_2)$ and $(f_1\otimes f_2\otimes f_3,
C_1+C_2+C_3)$ in $\sS_{2,\alpha}$.
\item For $(f,C)\in \sS_{k,\alpha}$ and any solid number $b$
with $\cpx{b}<(k+1)\alpha+3\log_3 2$, we include $(f\otimes x_1 + b, C+\cpx{b})$
in $\sS_{k+1,\alpha}$.
\item For $(f,C)\in \sS_{k,\alpha}$, any solid number $b$
with $\cpx{b}<(k+1)\alpha+3\log_3 2$, and any $v\in B_\alpha$, we include
$(v(f\otimes x_1 + b), C+\cpx{b}+\cpx{v})$ in $\sS_{k+1,\alpha}$.
\item For all $n\in T_\alpha$, we include $(n,\cpx{n})$ in $\sS_{k+1,\alpha}$.
\item For all $n\in T_\alpha$ and $v\in B_\alpha$, we include $(vn,\cpx{vn})$ in
$\sS_{k+1,\alpha}$.
\end{enumerate}

This is a finite set, as the $\sS_i$ for $i\le k$ are all finite, $B_\alpha$ is
finite, $T_\alpha$ is finite, and there are only finitely many $b$ satisfying
$\cpx{b}<(k+1)\alpha+3\log_3 2$, as this implies that
\[3\log_3 b<(k+1)\alpha+3\log_3 2.\]
Also, all elements of $\sS_{k+1,\alpha}$ have degree at most $k$: In case (1),
if $k+1>2$, $f$ and $g$ have degree at most $i-1$ and and $j-1$ respectively, so
$f\otimes g$ has degree at most $i+j-2=k$, while if $k+1=2$, then $f_1, f_2,$
and $f_3$ all have degree $0$, so $f_1 \otimes f_2$ and $f_1 \otimes f_2 \otimes
f_3$ also have degree $0$.  In cases (2) and (3), $f$ has degree at most $k-1$,
so $f\otimes x_1 + b$ has degree at most $k$.  Finally, in cases (4) and (5), we
are adding low-defect pairs of degree $0$.

So suppose that $N\in B_{(k+1)\alpha}$; we apply Theorem~\ref{themethod}.

In case (1) of Theorem~\ref{themethod}, if $k+1>2$, then there is a good
factorization $N=uv$ where $u\in B_{i\alpha}$, $v\in B_{j\alpha}$ with $i+j=k+2$
and $2\le i, j\le k$.  So by the inductive hypothesis, we can take $(f,C)\in
\sS_{i,\alpha}$ and $(g,D)\in \sS_{j,\alpha}$ such that $(f,C)$ efficiently
$3$-represents $u$ and $(g,D)$ efficiently $3$-represents $v$.  Since the
factorization $N=uv$ is good, it follows that $(f\otimes g,C+D)$ efficiently
represents $N$.  If $k+1=2$, there is either a good factorization $n=u_1 u_2$ or
a good factorization $n=u_1 u_2 u_3$ with all $u_\ell\in{B}_\alpha$.  So take
$(f_\ell,C_\ell)\in \sS_{1,\alpha}$ such that $(f_\ell,C_\ell)$ efficiently
$3$-represents $u_l$; then either $(f_1\otimes f_2, C_1+C_2)$ or $(f_1\otimes
f_2 \otimes f_3, C_1+C_2+C_3)$ efficiently $3$-represents $N$, as appropriate.

In case (2) of Theorem~\ref{themethod}, there are $a$ and $b$ with $N=a+b$,
$\cpx{N}=\cpx{a}+\cpx{b}$, $a\in A_{k\alpha}$, $b\le a$ a solid number, and
\[\dft(a)+\cpx{b}<(k+1)\alpha+3\log_3 2.\] In particular, we have
$\cpx{b}<(k+1)\alpha+3\log_3 2$.  Write $a=a' 3^\ell$ with $a'$ a leader and
$\cpx{a}=\cpx{a'}+3\ell$, so $a'\in B_{k\alpha}$, and pick $(f,C)\in
\sS_{k,\alpha}$ that efficiently $3$-represents $a'$.  Then $(f\otimes x_1 +
b,C+\cpx{b})$ is in $\sS_{k+1,\alpha}$ and efficiently $3$-represents $N$.  In
case (3) of Theorem~\ref{themethod}, there is a good factorization $n=(a+b)v$
with $v\in B_\alpha$ and $a$ and $b$ satisfying the conditions in the case (2)
of Theorem~\ref{themethod}, so the proof is similar; if we write $a=a' 3^\ell$
with $a'$ a leader and $\cpx{a}=\cpx{a'}+3\ell$ and pick $(f,C)\in
\sS_{k,\alpha}$ efficiently $3$-representing $a'$, then $(v(f\otimes
x_1+b),C+\cpx{b}+\cpx{v})$ efficiently $3$-represents $N$.

Finally, in cases (4) and (5) of Theorem~\ref{themethod}, the pair $(N,\cpx{N})$
is itself in $\sS_{k+1,\alpha}$, by cases (4) and (5) above.  This proves the
theorem.
\end{proof}

Note that while this theorem produces a covering of $B_r$, there is no guarantee
that for $f\in \sS_r$, all the numbers $3$-represented by $f$ will have defect
less than $r$; and in general this will not be the case.  For instance, if we
use the method of the proof of Theorem~\ref{mainthm} to produce the set $\sS_1$,
it will contain the polynomial $16x_1+1$, which $3$-represents the number $17$,
which has defect greater than $1$.  This deficiency will be remedied in a sequel
paper \cite{seq2}, where it will be shown how to choose the $\sS_r$ to get this
additional property.  There is also no guarantee that the numbers
$3$-represented by $f$ will be leaders; for instance, if we use this method to
produce the set $\sS_1$, it will also contain the constant polynomials $9$ and
$27$.

\subsection{Augmented low-defect polynomials}

Theorem~\ref{mainthm} gives us a representation of the leaders with defect less
than a fixed $r$, but we want to consider all numbers with defect less than $r$.
However, by Proposition~\ref{arbr}, any number can be written most-efficiently
as $3^k m$ for some $k\ge 0$ and some leader $m$.  To account for this, we
introduce the notion of an augmented low-defect polynomial:

\begin{defn}
For any low-defect polynomial $f$, we define $\xpdd{f}=f\otimes x$.  The
polynomial $\xpdd{f}$ will be called an \emph{augmented low-defect polynomial}.
For a low-defect pair $(f,C)$, the pair $(\xpdd{f},C)$ will be called an
\emph{augmented low-defect pair}.
\end{defn}

Note that augmented low-defect polynomials are never low-defect polynomials; by
Proposition~\ref{polystruct}, low-defect polynomials always have nonzero
constant term, while an augmented low-defect polynomial always has zero constant
term.

We can then make the following observations and definitions, parallel to the
contents of Subsections~\ref{3rep} and \ref{mainthmsec}:

\begin{cor}
\label{augbasicub}
If $(f,C)$ is a low-defect pair of degree $r$, then
\[\cpx{\xpdd{f}(3^{n_1},\ldots,3^{n_{r+1}})}\le C+3(n_1+\ldots+n_{r+1}).\]
\end{cor}

\begin{proof}
This is immediate from Proposition~\ref{basicub} and Lemma~\ref{powers}.
\end{proof}

\begin{defn}
Given a low-defect pair $(f,C)$ (say of degree $r$) and a number $N$,
we will say $(\xpdd{f},C)$ efficiently $3$-represents $N$ if there exist
$n_1,\ldots,n_{r+1}$ such that $N=\xpdd{f}(3^{n_1},\ldots,3^{n_{r+1}})$ and
$\cpx{N}=C+3(n_1+\ldots+n_{r+1})$.  More generally, we will also say $\xpdd{f}$
$3$-represents $N$ if there exist $n_1,\ldots,n_{r+1}$ such that
$N=\xpdd{f}(3^{n_1},\ldots,3^{n_{r+1}})$.
\end{defn}

\begin{cor}
\label{augdftbd}
Let $(f,C)$ be a low-defect pair of degree $r$, and let $n_1,\ldots,n_r$ be
nonnegative integers. Then
\[ \dft(\xpdd{f}(3^{n_1},\ldots,3^{n_{r+1}}))\le \dft_{f,C}(n_1,\ldots,n_r)\]
and the difference is an integer.
\end{cor}

\begin{proof}
This inequality is just Corollary~\ref{augbasicub} with
$3\log_3\xpdd{f}(3^{n_1},\ldots,3^{n_{r+1}})$ subtracted off both sides.  And
since Corollary~\ref{augbasicub} is an inequality of integers, the difference
is an integer.
\end{proof}

\begin{thm}
\label{augmainthm}
For any real $r\ge 0$, there exists a finite set $\sS_r$ of low-defect pairs
satisfying the following conditions:
\begin{enumerate}
\item Each $(f,C)\in \sS_r$ has degree at most $\lfloor r \rfloor$;
\item  For every $N\in A_r$, there exists some $(f,C)\in \sS_r$ such that
$(\xpdd{f},C)$ that efficiently $3$-represents $N$.
\end{enumerate}
\end{thm}

\begin{proof}
This is immediate from Theorem~\ref{mainthm} and Proposition~\ref{arbr}.
\end{proof}

\section{Facts from order theory and topology}
\label{external}

This section collects facts about well orderings and partial orderings
needed to prove the main result. 
Recall that a {\em well partial order} is a partial order which is well-founded (has
no infinite descending chains) and has no infinite antichains.  Any
totally-ordered extension of a well partial order is well-ordered.  Given a well
partial order $X$, we can consider the set of order types of well-orders
obtained by extending the ordering on $X$.  It was proved by D.H.J.~De Jongh and
R.~Parikh \cite[Theorem 2.13]{wpo} that for any well partial order $X$, the set
of ordinals obtained this way has a maximum; this maximum is denoted $o(X)$.
They further proved \cite[Theorem 3.4, Theorem 3.5]{wpo}:

\begin{thm}
\label{naturalops}
Let $X$ and $Y$ be two well partial orders.  Then $X\amalg Y$ and $X\times Y$
are well partial orders, and $o(X\amalg Y)=o(X)\oplus o(Y)$, and $o(X\times
Y)=o(X)\otimes o(Y)$, where $\oplus$ and $\otimes$ are the operations of natural
sum and natural product (also known as the Hessenberg sum and Hessenberg
product).
\end{thm}

The natural sum and natural product are defined as follows \cite{wpo}:
\begin{defn}
The \emph{natural sum} (also known as the \emph{Hessenberg sum}) of two ordinals
$\alpha$ and $\beta$, here denoted $\alpha \oplus \beta$, is defined by simply
adding up their Cantor normal forms as if they were ``polynomials in $\omega$''.
That is to say, if there are ordinals $\gamma_0 < \ldots < \gamma_n$ and whole
numbers $a_0, \ldots, a_n$ and $b_0, \ldots, b_n$ such that
$\alpha = \omega^{\gamma_n}a_n + \ldots + \omega^{\gamma_0}a_0$ and $\beta =
\omega^{\gamma_n}b_n + \ldots + \omega^{\gamma_0}b_0$, then
\[ \alpha\oplus\beta = \omega^{\gamma_n}(a_n+b_n) + \ldots + 
	\omega^{\gamma_0}(a_0+b_0). \]

Similarly, the \emph{natural product} (also known as the \emph{Hessenberg
product}) of $\alpha$ and $\beta$, here denoted $\alpha \otimes \beta$, is
defined by multiplying their Cantor normal forms as if they were ``polynomials
in $\omega$'', using the natural sum to add the exponents.  That is to say,
if we write $\alpha = \omega^{\gamma_n} a_n + \ldots + \omega^{\gamma_0} a_0$
and $\beta = \omega^{\delta_m}b_m + \ldots + \omega^{\delta_0}b_0$ with
$\gamma_0 < \ldots < \gamma_0$ and $\delta_0 < \ldots < \delta_m$ ordinals and
the $a_i$ and $b_i$ whole numbers, then
\[ \alpha\otimes\beta = \bigoplus_{\substack{0 \le i \le n \\ 0\le j \le m}}
	\omega^{\gamma_i \oplus \delta_j} a_i b_j. \]
\end{defn}

These operations are commutative and associative, and $\otimes$ distributes over
$\oplus$.  The expression $\alpha \oplus \beta$ is strictly increasing in
$\alpha$ and $\beta$; and $\alpha \otimes \beta$ is strictly increasing in
$\beta$ so long as $\alpha\ne 0$, and vice versa \cite{carruth}.

There are other definitions of these operations.  Given ordinals $\alpha$ and
$\beta$, $\alpha \oplus \beta$ is sometimes defined as $o(\alpha\amalg\beta)$,
and $\alpha \otimes \beta$ as $o(\alpha\times\beta)$, where for this definition
we consder $\alpha$ and $\beta$ as partial orders).  As noted above, De Jongh
and Parikh showed the stronger statement Theorem~\ref{naturalops}, from which it
follows that
\begin{eqnarray*}
o(\alpha_1 \amalg \ldots \amalg \alpha_n) & = &
	\alpha_1 \oplus \ldots \oplus \alpha_n \\
o(\alpha_1 \times \ldots \times \alpha_n) & = &
	\alpha_1 \otimes \ldots \otimes \alpha_n
\end{eqnarray*}

There is also a recursive definition \cite{ONAG}.

Note also the following statements about well partial orderings:

\begin{prop}
\label{image}
Suppose that $X$ is a well partially ordered set, $S$ a totally ordered set, and
$f:X\to S$ is monotonic.  Then $f(X)$ is well-ordered, and has order type at
most $o(X)$.
\end{prop}

\begin{proof}
Pick a well-ordering extending the ordering $\leq$ on $X$; call it $\preceq$.
Define another total ordering on $X$, call it $\leq'$, by $a <' b$ if either
$f(a) < f(b)$ or $f(a)=f(b)$ and $a\prec b$.  Observe that $\leq'$ is an
extension of $\leq$ as $f$ is monotonic, so it is a well-ordering and has order
type at most $o(X)$.  Since $f$ is clearly also monotonic when we instead use
the ordering $\leq'$ on the domain, its image is therefore also well-ordered
and of order type at most $o(X)$.
\end{proof}

Note in particular that if $X$ is the union of $X_1,\ldots,X_n$, then $o(X)\le
o(X_1)\oplus\ldots\oplus o(X_n)$ as $X$ is a monotonic image of
$X_1\amalg\ldots\amalg X_n$.  So we have:

\begin{prop}
\label{cutandpaste}
We have:
\begin{enumerate}
\item If $S$ is a well-ordered set and $S=S_1\cup\ldots\cup S_n$, and $S_1$
through $S_n$ all have order type less than $\omega^k$, then so does $S$.
\item If $S$ is a well-ordered set of order type $\omega^k$ and
$S=S_1\cup\ldots\cup S_n$, then at least one of $S_1$ through $S_n$ also has
order type $\omega^k$.
\end{enumerate}
\end{prop}

\begin{proof}
For (1), observe that the order type of $S$ is at most the natural sum of those
of $S_1,\ldots,S_n$, and the natural sum of ordinals less than $\omega^k$ is
again less than $\omega^k$.

For (2), by (1), if $S_1,\ldots, S_k$ all had order type less than $\omega^k$,
so would $S$; so at least one has order type at least $\omega^k$, and it
necessarily also has order type at most $\omega^k$, being a subset of $S$.
\end{proof}

For the proof of the main result we  will  also need some facts about 
well-ordered sets sitting inside the real numbers.
In particular, we need results about  closures and limit points of
such sets, with the ambient space carrying the order topology.  Since we have
not found all the following results in the literature, we supply proofs.

\begin{prop}
\label{closure}
Let $X$ be a totally ordered set, and let $S$ be a well-ordered subset of order
type $\alpha$.  Then $\overline{S}$ is also well-ordered, and has order type
either $\alpha$ or $\alpha+1$.  If $\alpha=\gamma+k$ where $\gamma$ is a limit
ordinal and $k$ is finite, then $\overline{S}$ has order type $\alpha+1$ if and
only if the initial segment of $S$ of order type $\gamma$ has a supremum in $X$
which is not in $S$.
\end{prop}

\begin{proof}
We induct on $\alpha$.  If $\alpha=0$, $S$ is empty and thus so is
$\overline{S}$.

If $\alpha=\beta+1$, say $x$ is the maximum element of $S$ and
$T=S\setminus \{ x \}$. Then $\overline{S}=\overline{T} \cup \{ x \}$, and $x$
is the maximum element of $\overline{S}$.  If $x\in\overline{T}$, then
$\overline{S}=\overline{T}$; otherwise its order type is $1$ greater.  So as
$\overline{T}$ has order type either $\beta$ or $\beta+1$ by the inductive
hypothesis, $\overline{S}$ has order type $\beta$, $\beta+1=\alpha$, or
$\beta+2=\alpha+1$.  Of course, the first of these is impossible, as its order
type must be at least $\alpha$, since it contains $S$, so the order type is
either $\alpha$ or $\alpha+1$.

Furthermore, if $\beta=\gamma+k$ where $\gamma$ is a limit ordinal, we can let
$R$ be the initial segment of $T$ (equivalently, of $S$) of order type $\gamma$.
Then by the inductive hypothesis, $\overline{T}$ has order type $\beta+1$ if and
only if $R$ has a supremum in $X$ which is not in $T$.  In the case where
$x\notin \overline{T}$, then $x\notin \overline{R}$ and so $x$ cannot be a
supremum of $R$ in $X$.  Hence, in this case, $\overline{T}$ has order type
$\beta+1$ if and only if $R$ has a supremum in $X$ which is not in $S$, and so
$\overline{S}$ has order type $\beta+2=\alpha+1$ if and only if $R$ has a
supremum in $X$ which is not in $S$.

In the case where $x\in\overline{T}$, it must be that $x$ is a supremum of $T$
in $X$.  Since $x$ is not itself in $T$, this requires that $\beta$ be a limit
ordinal, and hence that $\beta=\gamma$, i.e.~$T=R$, since $\gamma$ is the
largest limit ordinal smaller than $S$.  So $R$ has a supremum which is not in
$T$, namely, $x$; and so by the inductive hypothesis $\overline{T}$ has order
type $\beta+1$.  As $\overline{S}=\overline{T}$ in this case, it too has order
type $\beta+1=\alpha$.  Furthermore, $R$ has a supremum, $x$, but this supremum
is in $S$; thus the theorem is true in this case.

Finally we have the case where $\alpha$ is a limit ordinal.  If
$x\in\overline{S}$, either $x$ is an upper bound of $S$ or it is not; we will
first consider $R$, the subset of $\overline{S}$ consisting of those elements
which are not upper bounds of $S$.  For any $x\in R$, there is some $y\in
S$ with $y>x$, and so $x\in (-\infty,y)\cap\overline{S}$.  Since the former is
an open set, this means $x\in \overline{S\cap(-\infty,y)}$.  As $S\cap
(-\infty,y)$ is a proper initial segment of $S$, by the inductive hypothesis,
its closure is well-ordered.  Note that for varying $y$, the sets
$\overline{S\cap(-\infty,y)}$ form a chain under inclusion of well-ordered sets,
with smaller ones being initial segments of larger ones.  So as $R$ is the union
of these, it is well-ordered, and its order type is equal to their supremum.
Now clearly the order type of $R$ is at least $\alpha$, since $R$ includes $S$;
and by the inductive hypothesis, it is at most $\lim_{\beta<\alpha} (\beta+1) =
\alpha$.  So $R$ has order type $\alpha$.

This leaves the question of elements of $\overline{S}$ that are upper bounds of
$S$ (and hence $R$).  The only way such an element can exist is if it is the
supremum of $S$.  Hence, if $S$ has a supremum in $X$, and this supremum is not
already in $S$, then $\overline{S}$ has order type $\alpha+1$, and otherwise it
has order type $\alpha$.
\end{proof}

\begin{prop}
\label{initseg}
Suppose $X$ is a totally ordered set, $S$ a subset of $X$, and $T$ an initial
segment of $S$.  Then $\overline{T}$ is an intial segment of $\overline{S}$.
\end{prop}

\begin{proof}
Suppose $x\in \overline{T}$, $y\in\overline{S}$, and $y<x$; we want to show
$y\in \overline{T}$.  The set $(y,\infty)$ is an open subset of $X$ and
contains $x\in\overline{T}$, thus it also contains some $t\in \overline{T}$.
That is to say, there is some $t\in T$ with $t>y$.

Now say $U$ is any open neighborhood of $y$; then $U\cap (-\infty,t)$ is again
an open neighborhood of $y$, and since $y\in \overline{S}$, there must exist
some $s\in S\cap U\cap(-\infty,t)$.  But then $s\in S$, $s<t$, and $t\in T$, so
$s\in T$ as well as we assumed that $T$ was an initial segment of $S$.  Thus
each neighborhood $U$ of $y$ contains some element of $T$, that is to say, $y\in
\overline{T}$.
\end{proof}

\begin{cor}
\label{limitvsclos}
Let $X$ be a totally ordered set with the least upper bound property, and $S$ a
well-ordered subset of $X$ of order type $\alpha$.  Then if $\beta<\alpha$ is a
limit ordinal, the $\beta$'th element of $\overline{S}$ is the supremum (limit)
of the initial $\beta$ elements of $S$.
\end{cor}

\begin{proof}
Let $T$ be the intial segment of $S$ of order type $\beta$.  Since
$\beta<\alpha$, $T$ is bounded above in $S$, and thus in $X$, and thus it has a
supremum $s$.  This supremum $s$ is not in $T$ as $T$ has order type $\beta$, a
limit ordinal, and thus has no maximum.  So $\overline{T}$, by
Proposition~\ref{closure}, has order type $\beta+1$, and $s$ is clearly its
final element.  So by Proposition~\ref{initseg}, it is the $\beta$'th element of
$\overline{S}$ as well, and by definition it is the supremum of the initial
$\beta$ elements of $S$.
\end{proof}

\begin{prop}
\label{limpts1}
If $S$ is a well-ordered set of order type $\alpha < \omega^{n+1}$ with $n$
finite, then $S'$, the set of limit points of $S$ (in the order topology) has
order type strictly less than $\omega^n$.
\end{prop}

\begin{proof}
Since we are considering $S$ purely as a totally-ordered set and not embedded in
anything else, we may assume it is an ordinal.  Let $\beta$ be the order type of
$S'$.  The elements of $S'$ consist of the limit ordinals less than $\alpha$.
If $n=0$, then $\alpha$ is finite and so $\beta=0<\omega^0$.

Otherwise, $\alpha<\omega^{n+1}$ so say $\alpha \le \omega^n k$.  An ordinal
$\gamma$ is a limit ordinal if and only if it can be written as $\omega\gamma'$
for some $\gamma'>0$.  Since, assuming $n>0$, $\omega\gamma'<\omega^n k$ if and
only if $\gamma'<\omega^{n-1} k$, the order type of the set of limit ordinals
less than $\omega^n k$ is easily seen to be $\omega^{n-1}k-1$ (where the $1$ is
subtracted off the beginning; this only makes a difference if $n=1$).  So the
order type of $\beta$ is at most $\omega^{n-1}k-1<\omega^n$.
\end{proof}

It is not too hard to write down a general formula for the order type of $S'$ in
terms of the order type of $S$ (even without the restriction that
$\alpha<\omega^\omega$), but we will not need such detail here.  See 
\cite[Theorem 8.6.6]{semadeni} for more on this.

\begin{prop}
\label{limpts2}
Let $T$ be a totally-ordered set and $S$ a well-ordered subset.  If $S'$ (in the
order topology on $T$) has order type at least $\omega^n$ with $n$ finite, then
$S$ has order type at least $\omega^{n+1}$.
\end{prop}

\begin{proof}
Suppose $S$ has order type less than $\omega^{n+1}$.  Then by
Proposition~\ref{closure}, so does $\overline{S}$.  Since $\overline{S}'=S'$, we
can just consider $\overline{S}$.  And we can consider the order topology on
$\overline{S}$ instead of the subspace topology, since the former is coarser and
thus $\overline{S}$ has more limit points under it.  But by
Proposition~\ref{limpts1}, the order type of $\overline{S}'$ in the order
topology on $\overline{S}$ is less than $\omega^n$.  Hence $\overline{S}'$ under
the subspace topology also has order type less than $\omega^n$, and hence $S'$
has order type less than $\omega^n$.  So if $S'$ has order type at least
$\omega^n$, then $S$ has order type at least $\omega^{n+1}$.
\end{proof}

\section{Well-ordering of defects}
\label{secwo}

We now begin proving 
well-ordering theorems about defects.

\begin{prop}
\label{mono}
Let $(f,C)$ be a low-defect pair; then the function $\dft_{f,C}$ is strictly
increasing in each variable.
\end{prop}

\begin{proof}
Suppose $f$ has degree $r$.  We can define $g$, the reverse polynomial of $f$:
\[g(x_1,\ldots,x_r)=x_1\ldots x_r f(x_1^{-1},\ldots,x_r^{-1}).\]
So $g$ is a multilinear polynomial in $x_1, \ldots, x_r$, with the coefficient
of $\prod_{i\in S} x_i$ in $g$ being the coefficient of $\prod_{i\notin S} x_i$
in $f$.  By Proposition~\ref{polystruct}, $f$ has nonnegative coefficients, so
so does $g$; since the constant term of $f$ does not vanish, the $x_1\ldots x_r$
term of $g$ does not vanish. Hence $g$ is strictly increasing in each variable.

Then
\begin{eqnarray*}
\dft_{f,C}(n_1,\ldots,n_r)=C+3(n_1+\ldots+n_r)-3\log_3 f(3^{n_1},\ldots,3^{n_r})
\\ = C-3\log_3 \frac{f(3^{n_1},\ldots,3^{n_r})}{3^{n_1+\ldots+n_r}}
= C-3\log_3 g(3^{-n_1},\ldots,3^{-n_r})
\end{eqnarray*}
which is strictly increasing in each variable, as claimed.
\end{proof}

\begin{prop}
\label{wo1}
Let $(f,C)$ be a low-defect pair of degree $r$; then the image of $\dft_{f,C}$
is a well-ordered subset of $\mathbb{R}$, with order type $\omega^r$.
\end{prop}

\begin{proof}
By Proposition~\ref{mono}, $\dft_{f,C}$ is a monotonic function from
$\mathbb{Z}_{\ge0}^r$ to $\mathbb{R}$, and $\mathbb{R}$ is totally ordered, so
by Proposition~\ref{image} and Theorem~\ref{naturalops} its image is a
well-ordered set of order type at most $\omega^r$.

For the lower bound, we induct on $r$.  Let $S$ denote the image of
$\dft_{f,C}$. If $r=0$, $\dft_{f,C}$ is a constant and so $S$ has order type
$1=\omega^0$.  Now suppose $r\ge1$ and that this is true for $r-1$.  By
Lemma~\ref{maxvar}, we can write $f=h\otimes(g\otimes x_1+c)$ where $c$ is a
positive integer and $g$ and $h$ are low-defect polynomials.  Unpacking this
statement, if $s$ is the degree of $h$, we have
$f(x_1,\ldots,x_r)=h(x_1,\ldots,x_s)(g(x_{s+1},\ldots,x_{r-1})x_r+c)$.
Then
\begin{eqnarray*}
\dft_{f,C}(n_1,\ldots,n_r) & =& (C-\cpx{h})+\dft_{h}(n_1,\ldots,n_s) +\\
& &
3(n_{s+1}+\ldots+n_{r-1})-3\log_3(g(3^{n_{s+1}},\ldots,3^{n_{r-1}})+c3^{-n_r}).
\end{eqnarray*}
Thus,
\begin{eqnarray*}
\lim_{n_r\to\infty} \dft_{f,C}(n_1,\ldots,n_r) & =&
C-\cpx{h} + \dft_h(n_1,\ldots,n_s) +\\ & & 3(n_{s+1}+\ldots+n_{r-1}) -
3\log_3(g(3^{n_{s+1}},\ldots,3^{n_{r-1}})) \\
&=&
C-\cpx{h}-\cpx{g}+\dft_h(n_1,\ldots,n_s)+\dft_g(n_{s+1},\ldots,n_{r-1})\\
& = &
C -3\log_3(h(n_1,\ldots,n_s)g(n_{s+1},\ldots,n_{r-1})) \\
& = &
C-\cpx{g\otimes h}+\dft_{g\otimes h}(n_1,\ldots,n_{r-1}).
\end{eqnarray*}
And since $\dft_{f,C}$ is increasing in $n_r$, this means that this is in fact a
limit point of $S$.  So we see that $S'$ contains a translate of the image of
$\dft_{g\otimes h}$.  The degree of $g\otimes h$ is $r-1$, so by the inductive
hypothesis, this image has order type at least $\omega^{r-1}$.  Thus $S'$ has
order type at least $\omega^{r-1}$, and so by Proposition~\ref{limpts2}, this
means that $S$ has order type at least $\omega^r$.
\end{proof}

\begin{prop}
\label{wo2}
Let $(f,C)$ be a low-defect pair of degree $r$; then the set of $\dft(n)$ for
all $n$ $3$-represented by the augmented low-defect polynomial $\xpdd{f}$ is a
well-ordered subset of $\mathbb{R}$, with order type at least $\omega^r$ and at
most $\omega^r(\lfloor \delta(f,C) \rfloor+1)<\omega^{r+1}$.  The same is true
if $f$ is used instead of the augmented version $\xpdd{f}$.
\end{prop}

\begin{proof}
Let $S$ be the set of all $\dft(n)$ for all $n$ that are $3$-represented by
$\xpdd{f}$, and let $T$ be the image of $\dft_{f,C}$.  By Proposition~\ref{wo1},
$T$ is a well-ordered subset of $\mathbb{R}$, of order type $\omega^r$.  Suppose
$n=\xpdd{f}(3^{m_1},\ldots,3^{m_{r+1}})$.
Then by Corollary~\ref{augdftbd},
\[\delta(n)=\delta_{f,C}(m_1,\ldots,m_{r+1})-k\]
for some $k\ge 0$.
But $\delta_{f,C}(m_1,\ldots,m_{r+1})\le\delta(f,C)$ by Proposition~\ref{dftbd},
and since $\delta(n)\ge 0$, this implies $k\le \delta(f,C)$.  As $k$ is an
integer, this implies
\[k\in\{0,\ldots,\lfloor \delta(f,C) \rfloor\},\]
which is a finite set.  Let $\ell$ refer to the number $\lfloor \delta(f,C)
\rfloor$.

Thus, $S$ is covered by finitely many translates of $T$; more specifically, we
can partition $T$ into $T_0$ through $T_\ell$ such that
\[S= T_0 \cup (T_1 - 1) \cup \ldots \cup (T_\ell - \ell).\]
Then the $T_i$ all have order type at most $\omega^r$, and by
Proposition~\ref{cutandpaste} at least one has order type $\omega^r$.  Hence $S$
is well-ordered of order type at most $\omega^r (\lfloor \delta(f,C) \rfloor
+1)< \omega^{r+1}$ by Propositions~\ref{naturalops} and \ref{image}.  And by the
above reasoning, it also has order type at least $\omega^r$.

The proof for $f$ instead of $\xpdd{f}$ is similar.
\end{proof}

\begin{prop}
\label{rwo1}
For any $s>0$, the set $\mathscr{D}\cap[0,s)$ is a well-ordered subset of
$\mathbb{R}$ with order type at least $\omega^{\lfloor s \rfloor}$ and less than
$\omega^{\lfloor s \rfloor+1}$.
\end{prop}

\begin{proof}
By Theorem~\ref{augmainthm}, there exists a finite set $\sS_s$ of low-defect
polynomials of degree at most $\lfloor s \rfloor$ such that each $n\in A_s$ can
be $3$-represented by $\xpdd{f}$ for some $f\in \sS_s$. By
Proposition~\ref{wo2}, for each $f\in \sS$, the set of defects of numbers
$3$-represented by $\xpdd{f}$ is a well-ordered set of order type less than
$\omega^{\lfloor s \rfloor+1}$.  Since $\mathscr{D}\cap[0,s)$ is covered by a
finite union of these, it is also well-ordered of order type less than
$\omega^{\lfloor s \rfloor +1}$ by Proposition~\ref{cutandpaste}.

For the lower bound on the order type, if $0<s<1$, observe that
$0\in\mathscr{D}\cap[0,s)$.  Otherwise, let $k=\lfloor s \rfloor$ and consider
the low-defect polynomial
\[ f = (\ldots(((3x_1+1)x_2+1)x_3+1)\ldots)x_k+1. \]
We have $\cpx{f}\le 3+k$, so $\dft(f)\le k\le s$.  And since $k\ge 1$, by
Propostion~\ref{dftbd} the set of $\dft(n)$ for $n$ that are $3$-represented by
$f$ is contained in $\mathscr{D}\cap[0,s)$; while by Proposition~\ref{wo2}, it
has order type at least $\omega^k$, proving the claim.
\end{proof}

We can thus conclude:

\begin{thm}
\label{omgomg}
The set $\mathscr{D}$ is a well-ordered subset of $\R$, of order type
$\omega^\omega$.
\end{thm}

\begin{proof}
By Proposition~\ref{rwo1}, we see that each initial segment of $\mathscr{D}$ is
well-ordered, and with order type less than $\omega^\omega$; hence $\mathscr{D}$
is well-ordered, and has order type at most $\omega^\omega$.  Also by
Proposition~\ref{rwo1}, we can find initial segments of $\mathscr{D}$ with order
type at least $\omega^n$ for any $n\in\mathbb{N}$, so $\mathscr{D}$ has order
type at least $\omega^\omega$.
\end{proof}

We have now determined the order type of $\mathscr{D}$.  However, we have not
fully determined the order types of $\mathscr{D}\cap[0,s]$ for real numbers $s$.
Of course in general determining this is complicated, but we can answer the
question when $s$ is an integer:

\begin{thm}
\label{rwo2}
For any whole number $k\ne 1$, $\mathscr{D}\cap[0,k]$ is a well-ordered subset
of $\mathbb{R}$ with order type $\omega^k$, while $\mathscr{D}\cap[0,1]$ has
order type $\omega+1$.
\end{thm}

\begin{proof}
The order type of $\mathscr{D}\cap[0,k]$ is either the same as that of
$\mathscr{D}\cap[0,k)$, or that same order type plus $1$, depending on whether
or not $k\in \mathscr{D}$.  By Theorem~\ref{oldprops}, the only integral
elements of $\mathscr{D}$ are $0$ and $1$, so what remains is to determine the
order type of $\mathscr{D}\cap[0,k)$.  For $k=0$ this is clearly $1=\omega^0$,
making the statement true for $k=0$, so assume $k\ge 1$.

By Proposition~\ref{rwo1}, $\mathscr{D}\cap[0,k)$ is well-ordered and has order
type at least $\omega^k$.  However its order type is also equal to the supremum
of the order types of $\mathscr{D}\cap[0,r)$ for $r<k$, and by
Proposition~\ref{rwo1}, since $k$ is an integer, these are all less than
$\omega^k$.  Hence its order type is also at most $\omega^k$, and thus exactly
$\omega^k$.  Thus for $k\ge 1$, the order type of $\mathscr{D}\cap[0,k]$ is
exactly $\omega^k$, unless $k=1$, in which case it is $\omega+1$.
\end{proof}

Putting these together, we have the main theorem:
\begin{proof}[Proof of Theorem~\ref{frontpagethm}]
The first part is Theorem~\ref{omgomg}.  The second part follows from the proof
of Theorem~\ref{rwo2}, or from Theorem~\ref{rwo2} and the fact that $1$ is the
only nonzero defect which is also an integer.
\end{proof}

We will further discuss the order type of $\mathscr{D}\cap[0,s]$ when $s$ is not
an integer in a future paper \cite{seq3}.

\section{Variants of the main theorem}
\label{variants}

In this section, we prove several variants of the main theorem, all showing
$\omega^{\omega}$ well-ordering for various related sets.

We begin with proving the well ordering holds for the closure
$\overline{\mathscr{D}}$ of the defect set in $\R$.

\begin{prop} 
\label{thmbig}
The set $\overline{\mathscr{D}}$, the closure of the defect set, is
well-ordered, with order type $\omega^\omega$.  Furthermore, for an integer
$k\ge 1$, the order type of $\overline{\mathscr{D}}\cap[0,k]$ is $\omega^k+1$.
(And $k\in\overline{\mathscr{D}}$, so $k$ is the $\omega^k$'th element of
$\overline{\mathscr{D}}$).
\end{prop}

\begin{proof}
By Proposition~\ref{closure}, the set $\overline{\mathscr{D}}$ is well-ordered,
and its order type is $\omega^\omega$ since $\mathscr{D}$ is unbounded in
$\mathbb{R}$.  For the set $\overline{\mathscr{D}}\cap[0,k]$, observe that this
set is is the same as the closure of $\mathscr{D}\cap[0,k]$ within $[0,k]$, so
Proposition~\ref{closure} implies this has order type $\omega^k+1$ since $[0,k]$
has the least-upper-bound property.  And since by Proposition~\ref{closure}, for
$r<k$ the set $\overline{\mathscr{D}}\cap[0,r]$ has order type less than
$\omega^k$, the $\omega^k$'th element must be $k$ itself.
\end{proof}

The other variants of the main result include considering defect sets for
integers $n$ whose complexity $\cpx{n}$ falls in individual congruence classes
modulo $3$ and, in a separate direction, restricting to stable defects.
Furthermore results in both directions can be combined.  These defect sets are
all well-ordered by virtue of being contained in $\overline{\mathscr{D}}$, and
the issue is to show they have the appropriate order type.

To prove the main theorem, we needed to
know that given a low-defect pair $(f,C)$ of degree $k$, we have
$\cpx{f(3^{n_1},\ldots,3^{n_k})}\le C+3(n_1+\ldots+n_k)$.  In order to prove
these more detailed versions, as a preliminary result we  demonstrate that for certain low-defect
pairs $(f,C)$, equality holds for ``most'' choices of $(n_1,\ldots,n_k)$.
Indeed, we'll need an even stronger statement: Since
$\cpx{f(3^{n_1},\ldots,3^{n_k})}\le C+3(n_1+\ldots+n_k)$, it follows that also
\[\cpx{f(3^{n_1},\ldots,3^{n_k})}_\st\le C+3(n_1+\ldots+n_k),\]
and it's equality in this form that we'll need for ``most'' $(n_1,\ldots,n_k)$.

\begin{prop}
\label{dump}
Let $(f,C)$ be a low-defect pair of degree $k$ with $\dft(f,C)<k+1$.  Define its
``exceptional set'' to be
\[
S:=\{(n_1,\ldots,n_k): \cpx{f(3^{n_1},\ldots,3^{n_k})}_\st<C+3(n_1+\ldots+n_k)\}
\]
Then the set $\{\dft(f(3^{n_1},\ldots,3^{n_k})):(n_1,\ldots,n_k)\in S\}$ has
order type less than $\omega^k$.  In particular, the
set $\{\dft(f(3^{n_1},\ldots,3^{n_k})):(n_1,\ldots,n_k)\notin S\}$ has order
type at least $\omega^k$, and thus so does the set
\[ \{\dft(f(3^{n_1},\ldots,3^{n_k})):(n_1,\ldots,n_k)\in \mathbb{Z}^k_{\ge
0}\} \cap \mathscr{D}_\st^C. \]
\end{prop}

\begin{proof}
The set $S$ can be equivalently written as
\[ \{ (n_1,\ldots,n_k): \cpx{f(3^{n_1},\ldots,3^{n_k})}_\st \le
C+3(n_1+\ldots+n_k)-1 \}\]
and hence as
\[ \{ (n_1,\ldots,n_k): \dft_\st(f(3^{n_1},\ldots,3^{n_k}))\le
\dft_{f,C}(n_1,\ldots,n_k)-1 \}.\]

Hence for $(n_1,\ldots,n_k)\in S$, we have
\[\dft_\st(f(3^{n_1},\ldots,3^{n_k}))\le\dft(f,C)-1<k,\] and thus by
Proposition~\ref{rwo1}, the set of these stable defects has order type less than
$\omega^k$.

Equivalently, applying Proposition~\ref{cutandpaste}, the set
\[\{\dft(f(3^{n_1},\ldots,3^{n_k})):(n_1,\ldots,n_k)\in S\}\] and the set
$\dft_{f,C}(S)$ have order type less than $\omega^k$, since each is a finite
union of translates of subsets of the set
$\{\dft(f(3^{n_1},\ldots,3^{n_k})):(n_1,\ldots,n_k)\in S\}$.

So consider the set
\[\{\dft(f(3^{n_1},\ldots,3^{n_k})):(n_1,\ldots,n_k)\notin S\},\]
which can equivalently be written as
\[\{\dft_\st(f(3^{n_1},\ldots,3^{n_k})):(n_1,\ldots,n_k)\notin S\},\]
since for $(n_1,\ldots,n_k)\notin S$, the number $f(3^{n_1},\ldots,3^{n_k})$ is
stable.  This set must have order type at least $\omega^k$ by
Proposition~\ref{wo2} and Proposition~\ref{cutandpaste}.  Since for
$(n_1,\ldots,n_k)\notin S$, we have that $f(3^{n_1},\ldots,3^{n_k})$ is stable
and
\[\cpx{f(3^{n_1},\ldots,3^{n_k})} = C+3(n_1+\ldots+n_k)\equiv C\pmod{3},\]
this implies that the set
\[ \{\dft(f(3^{n_1},\ldots,3^{n_k})):(n_1,\ldots,n_k)\in \mathbb{Z}^k_{\ge 0}\}
\cap \mathscr{D}_\st^C, \]
being a superset of the above, has order type at least $\omega^k$.
\end{proof}

Recall that $\mathscr{D}_\st^a$ denotes the set of defect values $\delta(n)$
taken by stable numbers $n$ having complexity $\cpx{n} \equiv a ~(\bmod \,
3).$ Using the Proposition above, we can now prove:

\begin{thm}
\label{thmsmall} 
For $a=0,1,2$, the  stable defect sets $\mathscr{D}_\st^a$ are well-ordered, with order type
$\omega^\omega$.  Furthermore, if $k\equiv a\pmod{3}$, then the set
$\mathscr{D}_\st^a\cap[0,k]$ has order type $\omega^k$.
\end{thm}

\begin{proof}
Each of these sets is a subset of $\mathscr{D}$ and so they are well-ordered
with order type at most $\omega^\omega$.  To check that it is in fact exactly
$\omega^\omega$, consider the following low-defect polynomial:
\[ f_{a,k} := (\ldots(((ax_1+1)x_2+1)x_3+1)\ldots)x_k+1. \]
Specifically, consider the low-defect pair $(f_{a,k},\cpx{a}+k)$, for $a=2,3,4$.
Observe that $\dft(f_{a,k},\cpx{a}+k)=\dft(a)+k$, and for these choices of $a$,
we have $\dft(a)<1$.  Thus for $a=2,3,4$, $f_{a,k}$ satisfies the conditions of
Proposition~\ref{dump}.  Thus for $a=2,3,4$ and $k\ge 0$,
$\mathscr{D}_\st^{a+k}$ has order type at least $\omega^k$.  Since regardless of
$k$, the set $\{2+k,3+k,4+k\}$ is a complete system of residues modulo $3$, it
follows that for $a=0,1,2$ and any $k$, the set $\mathscr{D}_\st^a$ has order
type at least $\omega^k$.  Hence $\mathscr{D}_\st^a$ has order type at least
$\omega^\omega$ and hence exactly $\omega^\omega$.

Now suppose we take $k\equiv a\pmod{3}$.  We know, if $k\ne 1$, that
$\mathscr{D}_\st^a\cap[0,k]$ has order type at most $\omega^k$ by
Theorem~\ref{rwo2}.  (If $k=1$, we know this because $1\notin \mathscr{D}_\st$.)
To see that it is at least $\omega^k$, we consider the low-defect pair
$(f_{3,k},3+k)$.  Observe that $\dft(f_{3,k},3+k)=k$, and so (by
Proposition~\ref{dump}) the set $\mathscr{D}_\st^{3+k}\cap[0,k]$ has order type
at least $\omega^k$.  Since $3+k\equiv a\pmod{3}$, this is the same as the set
$\mathscr{D}_\st^{a}\cap[0,k]$, proving the claim.
\end{proof}

With this result in hand,  we can now prove:

\begin{thm}
\label{thmfinal}
We have:
\begin{enumerate}
\item The defect set $\mathscr{D}$ and stable defect set $\mathscr{D}_\st$ are both well-ordered, both with
order type $\omega^\omega$.  Furthermore, the set $\mathscr{D}_\st\cap[0,k]$ has
order type $\omega^k$, and for $k\ne1$, so does $\mathscr{D}\cap[0,k]$.
\item The sets $\overline{\mathscr{D}_\st}$ and $\overline{\mathscr{D}}$ are
well-ordered, both with order type $\omega^\omega$.  Furthermore, for $k\ge 1$,
the sets $\overline{\mathscr{D}_\st}\cap[0,k]$ and
$\overline{\mathscr{D}_\st}\cap[0,k]$ have order type $\omega^k+1$ (and both
contain $k$, so $k$ is the $\omega^k$'th element of both).
\item For $a=0,1,2$, the sets $\mathscr{D}^a$ and $\mathscr{D}^a_\st$ are all
well-ordered, each with order type $\omega^\omega$.  Furthermore, if $a\equiv
k\pmod{3}$, then $\mathscr{D}^a\cap[0,k]$ and $\mathscr{D}^a_\st\cap[0,k]$ have
order type $\omega^k$
\item For $a=0,1,2$, the sets $\overline{\mathscr{D}^a}$ and
$\overline{\mathscr{D}_\st^a}$ are well-ordered with order type $\omega^\omega$.
Furthermore, if $k\ge1$ and $a\equiv k\pmod{3}$, then
$\overline{\mathscr{D}^a}\cap[0,k]$ and $\overline{\mathscr{D}_\st^a}\cap[0,k]$
have order type $\omega^k+1$ (and each contains $k$, so $k$ is the $\omega^k$'th element).
\end{enumerate}
\end{thm}

\begin{proof}
The part of (1) for $\mathscr{D}$ is just Theorem~\ref{rwo2}. To prove the rest, observe that the
order type of $\mathscr{D}_\st$ is $\omega^\omega$ because it is contained in
$\mathscr{D}$ and contains, e.g., $\mathscr{D}_\st^0$.  For $k\ne 1$, we can see
that the order type of $\mathscr{D}_\st\cap[0,k]$ is at most $\omega^k$ because
it is contained in $\mathscr{D}\cap[0,k]$.  For $k=1$, we need to additionally
note that $1\notin \mathscr{D}_\st$.  Finally, the order type of
$\mathscr{D}_\st\cap[0,k]$ is at least $\omega^k$ because it contains
$\mathscr{D}_\st^k\cap[0,k]$.

The part of (2) for $\overline{\mathscr{D}}$ is  Proposition~\ref{thmbig}.  To prove the rest, note that by
(1), $\mathscr{D}_\st$ is unbounded in $\mathbb{R}$, and so
Proposition~\ref{closure} implies that $\overline{\mathscr{D}_\st}$ is
well-ordered with order type $\omega^\omega$.  For
$\overline{\mathscr{D}_\st}\cap[0,k]$, (1) together with
Proposition~\ref{closure} implies this has order $\omega^k+1$.  And since by
Proposition~\ref{closure}, for $r<k$ the set
$\overline{\mathscr{D}_\st}\cap[0,r]$ has order type less than $\omega^k$, the
$\omega^k$'th element must be $k$ itself.

The part of (3) for  $\overline{\mathscr{D}^a}$ is just Theorem~\ref{thmsmall}.  To prove the rest, observe that the
sets $\mathscr{D}^a$ are well-ordered with order type $\omega^\omega$ because
they contain $\mathscr{D}_\st^a$ and are contained in $\mathscr{D}$.
Furthermore, if $a\equiv k\pmod{3}$, then $\mathscr{D}^a\cap [0,k]$ has order
type at least $\omega^k$ by Theorem~\ref{thmsmall}.  If $k\ne 1$, then
Theorem~\ref{rwo2} shows it has order type at most $\omega^k$; for $k=1$, we
need to additionally note that $1\notin \mathscr{D}^a$.

Finally, to prove (4), note that by Theorem~\ref{thmsmall} and (3),
$\mathscr{D}^a$ and $\mathscr{D}_\st^a$ are unbounded in $\mathbb{R}$, and so
Proposition~\ref{closure} implies $\overline{\mathscr{D}_\st^a}$ and
$\overline{\mathscr{D}^a}$ are well-ordered with order type $\omega^\omega$.
For $\overline{\mathscr{D}^a_\st}\cap[0,k]$ and
$\overline{\mathscr{D}^a}\cap[0,k]$, Theorem~\ref{thmsmall} and (3) together
with Proposition~\ref{closure} imply these have order type $\omega^k+1$.  And
since by Proposition~\ref{closure}, for $r<k$ the sets
$\overline{\mathscr{D}^a_\st}\cap[0,r]$ and $\overline{\mathscr{D}^a}\cap[0,r]$
has order type less than $\omega^k$, the $\omega^k$'th element must be $k$
itself.
\end{proof}

We can also re state this result in the following way:
\begin{cor}
We have:
\begin{enumerate}
\item For $k\ge 1$, the $\omega^k$'th elements of $\overline{\mathscr{D}}$ and
$\overline{\mathscr{D}_\st}$ are both $k$.  If $a\equiv k\pmod{3}$, this is also
true of $\overline{\mathscr{D}^a}$ and $\overline{\mathscr{D}^a_\st}$.
\item For $k\ge 0$, the supremum of the initial $\omega^k$ elements of
$\mathscr{D}$ is $k$, and so is that of the initial $\omega^k$ elements of
$\mathscr{D}_\st$.  If $a\equiv k\pmod{3}$, then this is also true of
$\mathscr{D}^a$ and $\mathscr{D}^a_\st$.
\end{enumerate}
\end{cor}

\begin{proof}
Part (1) is just Theorem~\ref{thmfinal}.  Part (2), for $k\ge 1$, is
Theorem~\ref{thmfinal} and Corollary~\ref{limitvsclos}.  For $k=0$, this is
just the observation that $0$ is the intial element of $\mathscr{D}$ and so
also of $\mathscr{D}_\st$, $\mathscr{D}^0$, and $\mathscr{D}^0_\st$ (since
these all contain $0$).
\end{proof}

So we have now 
exhibited sixteen particular sets of defects that are well-ordered with
order type $\omega^\omega$: $\mathscr{D}$, $\mathscr{D}_\st$, the closures of
these sets, and for $a=0,1,2$, the sets $\mathscr{D}^a$, $\mathscr{D}^a_\st$,
and their closures.  We leave it for future work to resolve which of these sets
are distinct.

\subsection*{Acknowledgements}

Work of the author was supported by NSF grants DMS-0943832 and DMS-1101373.
The author thanks
J.~Arias de Reyna for suggested improvements and simplified proofs of
Proposition~\ref{initseg} and part of Proposition~\ref{wo1}.  The author thanks
J.~C.~Lagarias for help with editing and for suggesting references.  He thanks
Andreas Blass for suggesting references and further help with editing.
He is grateful to Joshua Zelinsky for much helpful discussion at the initial
stages of this work.

\appendix

\section{Conjectures of J. Arias de Reyna}
\label{secarias}

In his paper ``Complejidad de los n\'umeros naturales,'' \cite{Arias} Juan Arias
de Reyna proposed a series of conjectures about integer complexity.  These
conjectures also proposed a structure to integer complexity described by
ordinal numbers, but using a different language.  These conjectures
make assertions similar in spirit to some of the above results. Below we prove
modified versions of his conjectures 5 through 7.

The conjectures deal with the quantity $n3^{-\floor{\cpx{n}/3}}$, which is
related to (in fact, determined by) the quantity $\dft(n)$.  We recall first the
formula for the largest number writable with $k$ ones which was proved by
Selfridge (see
\cite{Guy}).

\begin{defn}
Let $E(k)$ denote the largest number writable with $k$ ones, i.e., the largest
number with complexity at most $k$.
\end{defn}

\begin{thm}[Selfridge]
The number $E(k)$ is given by the following formulae:
\begin{eqnarray*}
E(1) &=& 1\\
E(3j) &=& 3^j\\
E(3j+2) &= & 2 \cdot 3^j\\
E(3j+4) &=& 4 \cdot 3^j
\end{eqnarray*}
\end{thm}

Based on this, in \cite{paper1}, this author and Zelinsky noted:
\begin{prop}
\label{dRformulae}
We have $\delta(1) =1$ and 
\begin{displaymath}
\delta(n)=\left\{ \begin{array}{ll}
3\log_3 \frac{E(\cpx{n})}{n}	& \mathrm{if}\quad \cpx{n}\equiv 0\pmod{3}, \\
3\log_3 \frac{E(\cpx{n})}{n} +2\,\delta(2)
	& \mathrm{if}\quad \cpx{n}\equiv 1\pmod{3},  \,\,  \mathrm{with} \; n >  1, \\
3\log_3 \frac{E(\cpx{n})}{n} +\delta(2)
	& \mathrm{if}\quad \cpx{n}\equiv 2\pmod{3}. 
\end{array} \right.
\end{displaymath}
\end{prop}

That is to say, for $n>1$, given the congruence class of $\cpx{n}$ modulo $3$,
the quantity $nE(\cpx{n})^{-1}$ is a one-to-one and order-reversing function of
$\dft(n)$.

As noted above, whereas this author and Zelinsky considered $nE(\cpx{n})^{-1}$,
Arias de Reyna considered $n3^{-\floor{\cpx{n}/3}}$.  However, this is much the
same thing:

\begin{prop}
\label{theconstants}
For $k>1$,
\[ E(k) = c 3^\floor{\frac{k}{3}} \]
where
\begin{displaymath}
c=\left\{ \begin{array}{ll}
1	& \mathrm{if}\quad k\equiv 0\pmod{3}, \\
4/3 & \mathrm{if}\quad k\equiv 1\pmod{3}, \\
2 & \mathrm{if}\quad k\equiv 2\pmod{3}.
\end{array} \right.
\end{displaymath}
\end{prop}

So for $n>1$, within each congruence class of $\cpx{n}$ modulo $3$, the quantity
$n3^{-\floor{\cpx{n}/3}}$ is also a one-to-one and order-reversing function of
$\dft(n)$, being the same as $nE(\cpx{n})^{-1}$ up to a constant factor.

This allows us to conclude the following result, which is a modified version of
what one gets if one combines Arias de Reyna's Conjectures 5, 6, and 7 with his
Conjectures 3 and 4.

\begin{thm}
\label{jadrconj}
{\rm (Modified Arias de Reyna Conjectures 5, 6, 7)}

For $a=0,1,2$, the sets
\[ \left\{ \frac{n}{3^{\lfloor \cpx{n}/3 \rfloor}} : \cpx{n}\equiv a \pmod{3},
~~n\  \mbox{stable} \right\} \]
are reverse well-ordered, with reverse order type $\omega^\omega$.

Equivalently, for $a=0,1,2$, so are the sets
\[ \left\{ \frac{n}{E(\cpx{n})} : \cpx{n}\equiv a \pmod{3},~~n\  \mbox{stable} \right\}. \]
\end{thm}

\begin{proof}
By Propositions~\ref{dRformulae} and \ref{theconstants}, each of these is the
image of some $\mathscr{D}^a_\st$ under an order-reversing function.
\end{proof}

\end{document}